\documentclass[a4paper,11pt]{amsart}

\usepackage[left=2.5cm,right=2.5cm,top=2.5cm,bottom=3cm]{geometry}

\usepackage{amsmath}
\usepackage{amssymb}
\usepackage{amsthm}
\usepackage[normalem]{ulem}
\usepackage{nicematrix}

\usepackage{pgf,tikz}
\usetikzlibrary{shapes,automata,positioning,arrows,calc}
\usepackage{color}
\usepackage{graphicx}
\usepackage{hyperref}
\usepackage{setspace}
\usepackage[version=3]{mhchem}
\usepackage[sort&compress,comma,square,numbers]{natbib}
\usepackage{booktabs}
\usepackage{enumitem,hyperref}
\usepackage{url}
\usepackage{mathabx}

\newcommand{\R}{\mathbb{R}}

\newcommand{\Z}{{\mathbb Z}}

\DeclareMathOperator{\rk}{rk}
\DeclareMathOperator{\im}{im}
\DeclareMathOperator{\diag}{diag}

\renewcommand{\k}{{\kappa}}
\renewcommand{\l}{{\ell}}

\newtheorem{theorem}{Theorem}[section]
\newtheorem{lemma}[theorem]{Lemma}
\newtheorem{proposition}[theorem]{Proposition}
\newtheorem{corollary}[theorem]{Corollary}
\newtheorem{conjecture}[theorem]{Conjecture}
\theoremstyle{definition}

\newtheorem{example}[theorem]{Example}

\begin{document}

	\title[Absence of Hopf Bifurcation]{Network reduction and absence of Hopf Bifurcations in  dual phosphorylation networks with three Intermediates}
	
	\author[E.~Feliu and N.~Kaihnsa]{Elisenda Feliu and Nidhi Kaihnsa}

\tikzset{every node/.style={auto}}
\tikzset{every state/.style={rectangle, minimum size=0pt, draw=none, font=\normalsize}}
\tikzset{bend angle=15}

\maketitle

\begin{abstract}
	Phosphorylation networks, representing the mechanisms by which proteins are phosphorylated at one or multiple sites,  are ubiquitous in cell signalling and display rich dynamics such as unlimited multistability. Dual-site phosphorylation networks are known to exhibit oscillations in the form of periodic trajectories, when phosphorylation and dephosphorylation occurs as a mixed mechanism: phosphorylation of the two sites requires one encounter of the kinase, while dephosphorylation of the two sites requires two encounters with the phosphatase. A still open question is whether a mechanism requiring two encounters for both phosphorylation and dephosphorylation also admits oscillations. 
 In this work we provide evidence in favor of the absence of oscillations of this network by precluding Hopf bifurcations in any reduced network comprising three out of its four intermediate protein complexes. Our argument relies on a novel network reduction step that preserves the absence of Hopf bifurcations, and on a detailed analysis of the semi-algebraic conditions precluding Hopf bifurcations obtained from Hurwitz determinants of the characteristic polynomial of the Jacobian of the system. We conjecture that the removal of certain reverse reactions appearing in Michaelis-Menten-type mechanisms does not have an impact on the presence or absence of Hopf bifurcations. We prove an implication of the conjecture under certain favorable scenarios and support the conjecture with additional example-based evidence.

	\vskip 0.1in
 
	\noindent
	{\bf Keywords:} Dual-phosphorylation network, Hopf bifurcation, Hurwitz matrices, semi-algebraic sets, oscillation.

 \medskip
	\noindent {\bf 2020 MSC:}{ 
	37N25, 
    92C40, 
	92C45, 
	}
	
\end{abstract}

\section{Introduction}

Protein phosphorylation networks are important mechanisms for cell signalling and control \cite{rendall-MAPK,qiao:oscillations}. Typically, phosphorylation networks involve a substrate $S$ that exists in different phosphorylation states, and phosphorylation and dephosphorylation events are catalyzed by one or multiple enzymes. A simple well-studied scenario comprises a single kinase $K$, responsible for catalyzing all phosphorylation events, and a single phosphatase $F$, catalyzing all dephosphorylation events.  
When the substrate has multiple phosphorylation sites, full (de)phosphorylation can proceed by means of different   mechanisms \cite{SH09}, depending for example on whether one encounter of the substrate and the kinase can lead to the phosphorylation of one or multiple sites (e.g. yielding distributive and processive mechanisms), or on whether phosphorylation needs to occur in a specific order or not (e.g. yielding sequential, cyclic, or unordered phosphorylation). 

Under the assumptions of mass-action kinetics, these mechanisms have been extensively studied e.g., \cite{CMS,conradi-shiu-review,CFM, G-distributivity, G-PNAS,rendall-feliu-wiuf,rendall-2site}.
Two main open questions concern, on one hand, the maximal number of positive steady states that the sequential and distributive  phosphorylation cycle with $n$ sites admits \cite{FHC14}, and on the other, the existence of oscillations, and in particular, Hopf bifurcations, in the sequential and distributive dual phosphorylation cycle. In this paper we focus on the latter, and bring forward new ideas to tackle computationally prohibitive approaches.

Specifically, we consider the following reaction network with ordered and distributive phosphorylation:
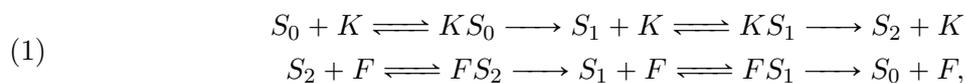
\begin{align}\label{eq:2sitefullnetwork}
\begin{split}
S_0 + K \ce{<=>} KS_0 \ce{->} S_1+K   \ce{<=>} KS_1 \ce{->} S_2+K \\
S_2 + F  \ce{<=>} FS_2 \ce{->} S_1+F  \ce{<=>} FS_1 \ce{->} S_0+F,
\end{split}
\end{align}
where $S_0,S_1,S_2$ refer to the substrate being phosphorylated in none, the first, or both sites, respectively, $K,F$ are the kinase and phosphatase respectively, and $KS_0,KS_1,FS_2,FS_1$ are protein complexes. This network admits three positive steady states, with two of them being asymptotically stable, for some choices of parameter values, and hence, it is known to exhibit multistationarity \cite{Markevich-mapk,FKWY,Wang:2008dc,rendall-feliu-wiuf}. 
If   the phosphorylation mechanism proceeds processively, that is, as
\[S_0 + K \ce{<=>} KS_0 \ce{->}   KS_1 \ce{->} S_2+K,\]
then the network has a unique positive steady state, and admits Hopf bifurcations (and thereby periodic solutions). 
The latter was first established  in \cite{SK}, and consequently in \cite{CMS}, the authors carry out a detailed algebraic analysis of the parameter space that gives rise to Hopf bifurcations. 
The similarities between the processive and the distributive mechanisms  makes one think that network~\eqref{eq:2sitefullnetwork} also should admit Hopf bifurcations. Furthermore, a two-layer cascade comprising the dual phosphorylation cycle in the second layer is known to admit oscillations \cite{qiao:oscillations}. Finally, 
network~\eqref{eq:2sitefullnetwork} has, apparently, the ingredients to give rise to Hopf bifurcations, as we detail later on, see also  \cite{weber:osc}. Despite this evidence and efforts, Hopf bifurcations have not been found. As a matter of fact, it has been shown in \cite{CFM} that if two intermediate protein complexes among the four $KS_0,KS_1,FS_2,FS_1$ are removed from \eqref{eq:2sitefullnetwork} and all reactions are irreversible, then the network does not have Hopf bifurcations, despite displaying the same key characteristics of  \eqref{eq:2sitefullnetwork}.

A simple Hopf bifurcation arises when a single pair of complex-conjugate eigenvalues of the Jacobian matrix of the system crosses the imaginary axis with a varying parameter. For stability, one often requires the 
remaining eigenvalues to have negative real part as the parameter changes. The question of precluding or establishing a simple Hopf bifurcation in mass-action kinetics translates to checking some semi-algebraic conditions 
due to Liu \cite{Liu}, see also  \cite{yang,GMS,KW}. These conditions are stated in terms of the determinants of the
Hurwitz matrices $H_1,\dots,H_{s-1}$ obtained from the characteristic polynomial of the Jacobian matrix  and its coefficient $a_s$ of the monomial with lowest degree.  The semi-algebraic criterion for finding Hopf bifurcations
requires the existence of parameters and steady states for which  $\det(H_1)>0,\dots,\det(H_{s-2})>0$, $\det(H_{s-1})=0$ and $a_s>0$. This criterion reduces the problem into checking the signs that some polynomials attain simultaneously. For network  \eqref{eq:2sitefullnetwork}, $\det(H_1)>0,\dots,\det(H_{s-2})>0$ hold as these polynomials have only positive coefficients. Additionally,
$\det(H_{s-1})$ and $a_s$ have terms both with positive and negative coefficients, and hence there are no apparent obstructions to achieve $a_s>0$ while $\det(H_{s-1})=0$. However, for the simplified networks with two intermediates explained above, it was shown that   $a_s>0$ implied $\det(H_{s-1})>0$, and hence the semi-algebraic system does not have a solution \cite{CFM}. 

In principle, the approach in \cite{CFM} should extend to the full network \eqref{eq:2sitefullnetwork}, but the computational cost is prohibitive. To bypass the limitations of heavy computations, we introduce a network simplification that preserves the absence of Hopf bifurcations.  We conjecture that reverse reactions in any motif of the form $y\cee{<=>} Y \cee{->} y'$ (of which there are four instances in \eqref{eq:2sitefullnetwork}) play no role in the determination of Hopf bifurcations and can be disregarded. We contribute towards the conjecture by showing that, under certain hypotheses on $y$, $y'$, and the rest of the network,  the absence of Hopf bifurcations in the reduced network implies the absence in the original network (Theorem~\ref{thm:yes}). 

Afterwards, we show that any subnetwork of network~\eqref{eq:2sitefullnetwork} with three intermediate complexes does not admit a Hopf bifurcation (Theorem~\ref{thm:Hopf}). 
As network~\eqref{eq:2sitefullnetwork} displays a symmetry 
 given by  the mapping $(K,F,S_0,S_1,S_2,KS_0,FS_2,KS_1,FS_1)$ to $(F,K,S_2,S_1, S_0,FS_2,KS_0,FS_1,KS_1)$, there are exactly two networks with three intermediate complexes that have to be inspected: 
\begin{itemize}
	\item The network   with intermediates $KS_0,KS_1,$ and $FS_2$:
	\begin{align}\label{eq:3-intmdnetwork2}\tag{$\mathcal{G}_1$}
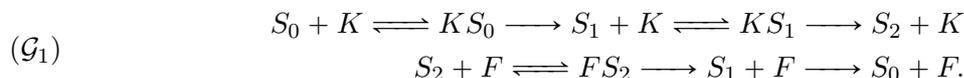
	\begin{split}
	S_0 + K \ce{<=>} KS_0 \ce{->} S_1+K   \ce{<=>} KS_1 \ce{->} S_2+K \\
	S_2 + F  \ce{<=>} FS_2 \ce{->} S_1+F   \ce{->} S_0+F.
	\end{split}
	\end{align}
	\item The network   with intermediates $KS_0,KS_1,$ and $FS_1$:
	\begin{align}\label{eq:3-intmdnetwork1}\tag{$\mathcal{G}_2$}
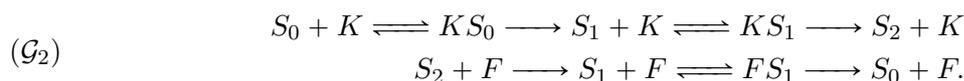
	\begin{split}
	S_0 + K \ce{<=>} KS_0 \ce{->} S_1+K  \ce{<=>} KS_1 \ce{->} S_2+K \\
	S_2 + F   \ce{->} S_1+F  \ce{<=>} FS_1 \ce{->} S_0+F.
	\end{split}
	\end{align}
\end{itemize}
Our proof relies on the network reduction in Theorem~\ref{thm:yes} and on symbolic computations performed on a computer. 
Our conclusion gives further evidence for \eqref{eq:2sitefullnetwork}
not admitting a Hopf bifurcation.
 
The article is structured as follows. In Section~\ref{sec:background} we give the background on reaction networks, convex parameters, and Hopf bifurcations, and establish the basic notation. 
In Section~\ref{sec:reductionmotif} we present 
a series of conjectures on reverse reactions of the special motif, state Theorem~\ref{thm:yes} covering a special case of the conjectures, and provide additional evidence to support the conjectures. 
In Section~\ref{sec:Hopfsimplifiednetwork} we show that the networks \eqref{eq:3-intmdnetwork2}, 
\eqref{eq:3-intmdnetwork1}, and all the versions obtained by making some reactions irreversible, do not admit Hopf bifurcations. 
Finally, in Section~\ref{sec:proof}, we prove Theorem~\ref{thm:yes} . The proof is technical and it is therefore kept in a separate section to ease the flow of the paper.

\subsection*{Acknowledgements}
This work has been funded by the Independent Research Fund of Denmark.

\subsection*{Supplementary File.} The computations for the proof of Theorem~\ref{thm:Hopf} can be found in the accompanying \texttt{Maple} files, downloadable from {\small \url{https://github.com/efeliu/HopfBifurcations2site}}.

\section{Background}\label{sec:background}
In this section we start by establishing the basic set up of reaction networks and  revisit the parametrization of the Jacobian matrices of a network evaluated at the steady states using convex parameters. More details can be found for example in \cite{feinberg-book}. Afterwards we review Hurwitz matrices and recall a sufficient condition based on their determinants for the absence of Hopf bifurcations.

\subsection{Reaction networks}\label{sec:network}

Given $n$ \textit{species} $X_1,\dots, X_n$, a complex is a non-negative integer linear combination $\sum_i\alpha_i X_i$ of the species, i.e. $\alpha_i\in \Z_{\geq 0}$ for all $i=1,\dots,n$.  A  (mass-action) \textit{reaction network} is a labelled digraph $G$ with complexes as vertices. The edges of the graph are called \textit{reactions} and are of the form 
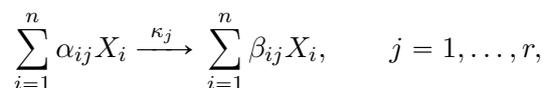
\[\sum_{i=1}^n\alpha_{ij} X_i\ce{->[\k_j]}\sum_{i=1}^n\beta_{ij} X_i, \qquad j=1,\dots, r, \]
where $\k_j>0$ for $j=1,\dots,r$ is called a \textit{reaction rate constant}. 
The set of reactions is assumed ordered, and the ordering is implicitly given by the subindex of the label. 
The source complex of a reaction is called the \emph{reactant} and the target complex the \emph{product}. 
We let $N=(\beta_{ij}-\alpha_{ij})\in \Z^{n\times r}$ be the \textit{stoichiometric matrix}, whose columns encode the net-production of the species in each reaction, and 
$B=(\alpha_{ij})\in \Z_{\geq 0}^{n\times r}$ the \emph{reactant} matrix, whose columns  encode the coefficients of the reactant.  
If $\alpha_{kj}=\beta_{kj}$ for all $j$, then 
$X_k$ is said to be a \textit{catalyst}. 

The concentration of species $X_i$ is denoted by $x_i$.  
Under the assumption of mass-action kinetics, the time evolution of the concentrations of the species is determined by a polynomial system of ordinary differential equations (ODE), which depends on  $\k\in \R^r_{>0}$, and is of the form
\begin{align}\label{eq:ODE}
\dot{x}=f_\k(x),  \quad x\in \R^n_{\geq 0}, \qquad \text{where} \quad f_\k(x)=  N \diag(\kappa) x^B.
\end{align}
Here we use the standard notation for the vector $x^B\in \R^{r}$ to be given by $(x^B)_j= \prod_{i=1}^{n}x_i^{\alpha_{ij}}$. 
The trajectories of \eqref{eq:ODE} are confined to parallel translates of $\im(N)$, and   $\R^n_{>0}$ and $\R_{\geq 0}^n$ are forward invariant by \eqref{eq:ODE} \cite{volpert}. 
A \textit{steady state} is a point $x^*\in \R_{\geq 0}^n$ such that 
$f_\k(x^*)=0$.

\subsection{Convex parameters}
For the study of Hopf bifurcations, we  need to consider the eigenvalues of the Jacobian matrix of \eqref{eq:ODE} evaluated at the positive steady states (i.e., steady states in $\R^n_{>0}$). The set of Jacobian matrices for all vectors of reaction rate constants and all positive steady states admits a parametrization in terms of the   \emph{convex parameters}. These parameters were first introduced by Clarke in \cite{Clarke} and are reviewed next.

Given a network $G$ with $n$ species and $r$ reactions,  the set $\ker(N)\cap \R_{\geq 0}^r$ is a convex pointed polyhedral cone, and hence admits a minimal generating set of extreme vectors $E_1,\ldots,E_m$ (which are unique up to scaling). We define  the \emph{extreme matrix} $E\in \R^{r\times m}$ to be the matrix with columns  $E_1,\ldots,E_m$. Although the matrix is not uniquely determined, any two such matrices differ in a rescaling and reordering of the columns. These operations have no effect in what follows, and hence, when referring to ``the extreme matrix'', we implicitly mean that a choice has been made.

For $h=(h_1,\ldots,h_n)\in \R_{> 0}^n$ and $\l=(\l_1,\ldots,\l_m)\in \R_{> 0}^m$, 
we consider the matrix 
\begin{align}\label{eqn:Jac-convex}
J(h,\l):=N \diag(E\l)B^\top\diag(h) \in \R^{n\times m} .
\end{align} 
If $E$ does not have a zero row, then from \cite[Eq (12) and Cor. 7(a)]{TF} follows that 
\begin{equation}\label{eq:convex}  
\{J(h,\l) : h \in \R_{> 0}^n,\ \l\in \R_{> 0}^m \} = \{J_{f_\k}(x) : f_\k(x)=0,\ \k\in \R^r_{>0},\ x\in \R^n_{>0} \},
\end{equation}
where $J_{f_\k}(x)=\Big(\tfrac{\partial f_{\k,i}}{\partial x_j}(x)\Big)$ is the Jacobian matrix of $f_\k$ evaluated at $x$. 
The vector $(h,\l) \in \R^n_{>0}\times\R_{> 0}^m$ is called the \emph{vector of convex parameters} throughout the article. 

In view of \eqref{eq:convex}, to study Hopf bifurcations we will focus on the characteristic polynomial of $J(h,\l)$  for all $h\in \R^n_{>0}$ and $\l \in \R_{> 0}^m$. 
For $s=\rk(N)$, given $(h,\l)$, the characteristic polynomial  $\det(z I_n-J(h,\l))$ of $J(h,\l)$ factors as
\begin{align*}
\det(z I_n-J(h,\l)) =z^{n-s} \big(  z^s+a_1 (h,\l) z^{s-1}+\ldots+a_s(h,\l)\big).
\end{align*}
Here $I_n$ is the identity matrix of size $n$. 
We will thus focus on the following polynomial
\begin{equation}\label{eq:charpoly}
    p_{(h,\l)} := z^s+a_1(h,\l)z^{s-1}+\ldots+a_s(h,\l) \quad \in \R[z],
\end{equation}
which we refer to   as the \emph{reduced characteristic polynomial} of $J(h,\l)$. By construction, the coefficients of $p_{(h,\l)}$ are polynomials  in $h,\l$.

\subsection{Hurwitz determinants}\label{sec:Hurwitz}
A (simple) Hopf bifurcation in a parametric ODE system with one parameter $\mu\in \R$
requires the existence of a parameter value  $\mu_0$ and a steady state $x_0$ such that the Jacobian matrix of the system at $x_0$ has a pair of purely imaginary eigenvalues, while the other eigenvalues have non-zero real part. It is also typically required that the other eigenvalues have negative real part. 
Therefore, one needs to study the roots of the reduced characteristic polynomial of the Jacobian matrix, and we will use a criterion based on the Hurwitz matrices. 
Specifically, for a univariate polynomial $p=a_0z^s+a_1z^{s-1}+\ldots+a_s$ and $i\in \{1,\dots,s\}$, the $i$-th \textit{Hurwitz matrix} is:
		\begin{align}
		H_i:=\begin{bmatrix}
		a_1&a_0&0&0&0&\ldots&0\\
		a_3&a_2&a_1&a_0&0&\ldots&0\\
		\vdots&\vdots&\vdots&\vdots&\vdots&&\vdots\\
		a_{2i-1}&a_{2i-2}&a_{2i-3}&a_{2i-4}&a_{2i-5}&\ldots&a_i
		\end{bmatrix} \in \R^{i\times i},
	\end{align}
where the $(k,l)$ entry is   $a_{2k-l}$ when $0\leq 2k-l\leq s$ and 0 otherwise. We refer to $\det(H_i)$ as the $i$-th \emph{Hurwitz determinant}.  Note that $\det(H_s)=a_s \det(H_{s-1})$.

The following proposition  is a direct consequence of \cite[Corollary 3.2]{KW} and \cite[Proposition 1]{CFM}. 

\begin{proposition}\label{cor:hurwitz}
For a reaction network  with $s=\rk(N)\geq 2$ and $(h,\l) \in \R^n_{>0}\times\R_{> 0}^m$, let  
$H_i(h,\l)$ be the $i$-th Hurwitz matrix of the reduced characteristic polynomial $p_{(h,\l)}$ of $J(h,\l)$ as defined in \eqref{eq:charpoly} and let $a_s(h,\l)$ be its constant term. 
Assume that for all $(h,\l) \in \R^n_{>0}\times\R_{> 0}^m$, it holds that
\begin{align*}
\det(H_i(h,\l))>0 &  \quad \text{ for  all } \quad i=1,\ldots,s-2, \\ 
  \det(H_{s-1}(h,\l))\neq 0 & \quad  \text{ if  }\quad a_s(h,\l) >0.
 \end{align*}
 Then for all $(h,\ell)\in \R^n_{>0}\times \R^m_{>0}$, $p_{(h,\l)}$ does not have a pair of purely imaginary roots. In particular  the parametric ODE system \eqref{eq:ODE} does not admit a simple Hopf bifurcation in $\R^n_{>0}$.
\end{proposition}

We note that the Hurwitz determinants  $\det(H_i(h,\l))$ can be regarded as polynomials in $h,\ell$, and when doing so we simply write 
$\det(H_i)\in \R[h,\ell]$. Then, Proposition~\ref{cor:hurwitz} expresses the problem of precluding Hopf bifurcations as 
identifying whether a semi-algebraic set is nonempty.

\section{Removing reversible reactions with intermediates and Hopf bifurcations}\label{sec:reductionmotif}

As biologically relevant networks  often have a large number of species and parameters, the Hurwitz determinants 
are large polynomials with many terms. 
 This makes the problem of checking the semi-algebraic conditions in Proposition~\ref{cor:hurwitz} computationally challenging. In these cases it is advantageous to exploit the structure of the network. A particularly recurring motif in relevant networks, including the well-known 
 Michaelis-Menten mechanism, is of the form
 \begin{equation}\label{motif2}
 y \ce{<=>} Y \ce{->} y',
\end{equation}
where $y,y'$ are complexes in species other than $Y$. The dual phosphorylation network in 
 \eqref{eq:2sitefullnetwork} has four such motifs, with the species $Y$ being one of $KS_0, KS_1, FS_2$ and $FS_1$.
 
 In this section we provide evidence for the reverse reaction $Y\ce{->} y$ having no influence on whether the network admits Hopf bifurcations or not, and show that this holds when extra conditions are imposed on $y,y'$ and the rest of the reactions of the network. 
 
 Given a reaction network $G$ containing a motif of the form \eqref{motif2}, we let $G'$ denote the reaction network obtained from $G$ by removing the reaction $Y\ce{->} y$. We will  refer to this operation  as ``removing the backward reaction'' or making the ``reversible reaction irreversible''. Note that $G$ and $G'$ have the same number of species. 
We conjecture that the following holds. 
 
 \begin{conjecture}\label{conj:1}
Consider a reaction network $G$  with $n$ species and having motif \eqref{motif2} such that the species $Y$ does not appear in any other complex of the network, and let $G'$ be the network obtained by removing the backward reaction of the motif. 
Then, $G$ admits a Hopf bifurcation in $\R^n_{>0}$ if and only if $G'$ admits a Hopf bifurcation in $\R^n_{>0}$. 
 \end{conjecture}
 
One way to approach Conjecture~\ref{conj:1} is to first compare the capacity of the characteristic polynomials associated with the two networks $G,G'$ to admit purely imaginary roots for some parameter values and steady state. 
In Lemma~\ref{lem:EmatrixG} it will be shown that if the extreme matrix $E$ of $G$ has $m$ columns, then the extreme matrix $E'$ of $G'$ has $m-1$ columns. 
We let $p_{(h,\l)}$ be the reduced characteristic polynomial of the Jacobian matrix $J(h,\ell)$ of $G$ in convex parameters $(h,\ell)\in \R^n_{>0} \times \R^{m}_{>0}$, c.f. \eqref{eqn:Jac-convex}, \eqref{eq:charpoly}, and 
$p_{(h,\l')}'$ be the reduced characteristic polynomial of the Jacobian matrix of $G'$,
\begin{align}\label{eqn:Jac-convexreduced}
J'(h,\l'):=N' \diag(E'\l')(B')^{\top}\diag(h), \qquad (h,\l')\in \R^n_{>0} \times \R^{m-1}_{>0},
\end{align}
where  $N'$ and $B'$ are the stoichiometric and reactant matrices of $G'$,  respectively.

 \begin{conjecture}\label{conj:2}
Consider a reaction network $G$ having motif \eqref{motif2} such that the species $Y$ does not appear in any other complex of the network. 
With the notation above, $p_{(h,\l)}$ has purely imaginary roots for some $(h,\l)\in \R^n_{>0}\times \R^m_{>0}$ if and only if  $p_{(h,\l')}'$ has purely imaginary roots for some $(h,\l')\in \R^n_{>0}\times \R^{m-1}_{>0}$.
 \end{conjecture}

 Conjecture~\ref{conj:2} does not directly imply Conjecture~\ref{conj:1}, as a Hopf bifurcation requires in addition that there is a pair of eigenvalues that crosses the imaginary axis as a parameter varies. However, if Conjecture~\ref{conj:2} is true, then a way to preclude Hopf bifurcations for $G$ is to preclude the existence of a pair of purely imaginary roots of $p_{(h,\l')}'$   for all $(h,\l')\in \R^n_{>0}\times \R^{m-1}_{>0}$. Since the coefficients of   $p_{(h,\l')}'$ depend on one less variable than those of $p_{(h,\l)}$, the computational cost is reduced. Additionally, for networks like  \eqref{eq:2sitefullnetwork}, with four such motifs, the reduction would be by four variables.
 
 One implication of Conjecture~\ref{conj:2} follows from our next conjecture, stating that the set of reduced characteristic polynomials for $G$ is contained in that of $G'$.

  \begin{conjecture}\label{conj:3}
Consider a reaction network $G$ having motif \eqref{motif2} such that the species $Y$ does not appear in any other complex of the network. 
With the notation above, 
\begin{equation}\label{eq:polys}
\big\{p_{(h,\l)} :  (h,\l)\in \R^n_{>0}\times \R^m_{>0}\big\}\subseteq \big\{p'_{(h,\l')} :  (h,\l')\in \R^n_{>0}\times \R^{m-1}_{>0}\big\}. 
\end{equation}
 \end{conjecture}

 In the remainder of this section, we first state  a scenario  where  Conjecture~\ref{conj:3} holds (Theorem~\ref{thm:yes}), and afterwards we discuss other networks, falling outside this scenario, that provide extra evidence of its generality. We will note in Example~\ref{ex::calcium} below that we cannot expect 
an equality of sets in \eqref{eq:polys}, as for the example the inclusion is strict. 
Observe that this does not necessarily imply that Conjecture~\ref{conj:2} is not true.

  \begin{theorem}\label{thm:yes} 
Let $G$ be a network that satisfies the following assumptions:
\begin{enumerate}[label=(\roman*)]
\item[(A1)] Contains  motif~\eqref{motif2} in the following form
\[ X_1+X_2\ce{<=>}  X_3 \ce{->} c +\delta X_2 \] 
where $X_i$ for $i=1,2,3$ are distinct, $\delta\in \{0,1\}$ and $c$ is a complex. 
\item[(A2)] $X_3$ only appears in the given motif.
\item[(A3)]  $c$ does not depend on $X_1,X_2,$ and $X_3$. 
\item[(A4)] $X_1$ is not a reactant of reactions outside the motif. 
\item[(A5)] If there are other reactions involving  $X_2$, then:
\begin{itemize}
\item[(a)] When $\delta=1$, if a reaction outside the motif has $X_2$  in the product, then $X_2$ is a catalyst and $X_1$ is not in the product. 
\item[(b)]  When $\delta=0$, the reactions outside the motif only have $X_2$ in the product.
\end{itemize} 
\end{enumerate}
Then, \[ \big\{p_{(h,\l)} :  (h,\l)\in \R^n_{>0}\times \R^m_{>0}\big\}\subseteq \big\{p'_{(h,\l')} :  (h,\l')\in \R^n_{>0}\times \R^{m-1}_{>0}\big\}. \]
 \end{theorem}

The proof of Theorem~\ref{thm:yes} is rather technical and is given in Section~\ref{sec:proof}. Importantly, the proof is constructive, that is, given $(h,\l)\in \R^n_{>0}\times \R^m_{>0}$, values $ (h,\l')\in \R^n_{>0}\times \R^{m-1}_{>0}$ are found such that 
  $p_{(h,\l)} = p'_{(h,\l')}$ (see Example~\ref{ex:constructive}). 
  The following corollary is an immediate consequence of Theorem~\ref{thm:yes}.

\begin{corollary}\label{cor:Hopf}
For a network $G$ satisfying conditions (A1)-(A5) of Theorem~\ref{thm:yes}, and with the notation above, if $p'_{(h,\l')}$ does not have purely imaginary roots for any $(h,\l')\in \R^n_{>0}\times \R^{m-1}_{>0}$, then neither $G$ nor $G'$ admits a Hopf bifurcation.
\end{corollary}

In view of Corollary~\ref{cor:Hopf}, in order to preclude Hopf bifurcations for $G$ and $G'$, it is enough to establish  that the Hurwitz determinants of the reduced characteristic polynomial of $J'(h,\ell')$ satisfy Proposition~\ref{cor:hurwitz}. 

   Theorem~\ref{thm:yes} will be relevant in the next section, when we preclude Hopf bifurcations for \eqref{eq:3-intmdnetwork2}. 
In the rest of this section, we discuss additional networks to show the use of  Theorem~\ref{thm:yes} and its limitations.

\begin{example}\label{ex:twoint} In \cite{CFM}, the authors considered four networks obtained by removing two intermediates from the dual phosphorylation network \eqref{eq:2sitefullnetwork} and all reverse reactions. As an application of Theorem~\ref{thm:yes} we consider the network $\mathcal{N}_1$ that exhibited interesting structure:
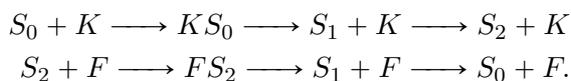
\begin{align*}
\begin{split}
S_0 + K \ce{->} KS_0 \ce{->} S_1+K    \ce{->} S_2+K \\
S_2 + F  \ce{->} FS_2 \ce{->} S_1+F   \ce{->} S_0+F.
\end{split}
\end{align*}
	The reduced characteristic polynomial of the Jacobian matrix of network $\mathcal{N}_1$ does not have purely imaginary roots and hence, the network does not admit Hopf bifurcations \cite[Theorem 2]{CFM}. Treating $S_0,K, KS_0$ as $X_1,X_2,X_3$ respectively in  Theorem~\ref{thm:yes}, the assumptions of the theorem hold with $\delta=1$. Hence, by Corollary~\ref{cor:Hopf}, we conclude that the following network does not admit a  Hopf bifurcation:
	
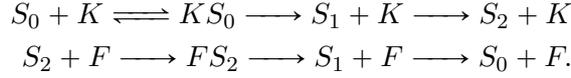
\begin{align*}
	\begin{split}
	S_0 + K \ce{<=>} KS_0 \ce{->} S_1+K    \ce{->} S_2+K \\ 
	S_2 + F  \ce{->} FS_2 \ce{->} S_1+F   \ce{->} S_0+F.
	\end{split}
	\end{align*}
We can now reiterate and consider $S_2,F,$ and $FS_2$ as $X_1,X_2,X_3$ and $\delta=1$ in the previous network, to conclude that the following network does not either admit a Hopf bifurcation:
	
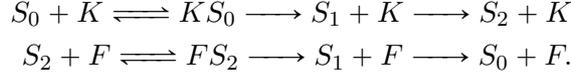
\begin{align*}
	\begin{split}
	S_0 + K \ce{<=>} KS_0 \ce{->} S_1+K    \ce{->} S_2+K \\ 
	S_2 + F  \ce{<=>} FS_2 \ce{->} S_1+F   \ce{->} S_0+F.
	\end{split}
	\end{align*}
 \end{example}

\begin{example}\label{ex::calcium}
We consider the simple calcium transport network studied in \cite{ges05}:
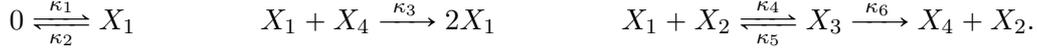
\begin{align*}
0 & \ce{<=>[\k_1][\k_2]}  X_1 \quad&
X_1+X_4 &  \ce{->[\k_3]} 2X_1\quad &
X_1+X_2 &  \ce{<=>[\k_4][\k_5]} X_3 \ce{->[\k_6]} X_4+X_2.
\end{align*}

Here, $X_1$ represents cytosolic calcium, $X_4$  calcium in the endoplasmic reticulum, and $X_2$ is an enzyme catalyzing the transfer via the formation of an intermediate protein complex $X_3$. This network is known to admit a Hopf bifurcation, see \cite{ges05}. 
The motif with reactions labeled by $\k_4,\k_5,\k_6$  is of the form \eqref{motif2} but does not satisfy assumption (A4) in Theorem~\ref{thm:yes}, as $X_1$ is a reactant in other reactions. 
However, removal of the reaction with label $\k_5$ yields a network that also admits a Hopf bifurcation for 
\[(\k_1, \k_2,\k_3,\k_4,\k_6)= (1,1,\tfrac{1}{2},1,1), \qquad (x_1,x_2,x_3,x_4)= (1,\tfrac{45}{7},\tfrac{45}{7},\tfrac{90}{7}),\] 
with $\k_4$ as the Hopf parameter. This gives additional support for the validity of Conjectures~\ref{conj:1} and \ref{conj:2}.
 The containment in Conjecture~\ref{conj:3} also holds, but is strict. To see this, we write the reduced characteristic polynomial of the full network as
\[p_{(h,\l)}:=z^3+a_1(h,\ell)z^2+a_2(h,\ell)z+a_3(h,\ell).\]
With the extreme matrix
{\small	\[E= \begin{bmatrix}
	1 & 1 & 0 & 0 & 0 & 0 \\
	0 & 0 & 1 & 1 & 0 & 1 \\
	0 & 0 & 0 & 1 & 1 & 0 
	\end{bmatrix}^\top,\]}
the coefficients of $p_{(h,\ell)}$ are:
{\small
\begin{align*}
	a_1(h,\ell)&:=h_1 \l_1 + h_1 \l_3 + h_2( \l_2 + \l_3) + h_3( \l_2 + \l_3) + h_4 \l_2,  \\
	a_2(h,\ell)&:=(h_1 \l_1+h_4 \l_2 - h_1  \l_2 )h_2 ( \l_2+ \l_3)     + (h_1 \l_1  + h_4 \l_2) h_3  (\l_2 + \l_3) + h_1 h_4 \l_1 \l_2 + h_1 h_4\l_2( \l_2+ \l_3), \\
	a_3(h,\ell)&:=h_1 h_2 h_4 \l_1 \l_2(\l_2 +\l_3) + h_1 h_3 h_4 \l_1 \l_2 (\l_2+ \l_3).
\end{align*}
}
The following definitions
\[ 
	t_1:=h_1\l_1,\quad t_2:=h_2(\l_2+\l_3),\quad t_3:=h_3(\l_2+\l_3),\quad t_4:=h_4\l_2, \quad t_5:=h_1\l_3, \quad t_6:=h_1\l_2,
\] 
 define a surjective map from $\R^7_{>0}$ to $\R^6_{>0}$, since given $t_1,\dots,t_6\in \R^6_{>0}$, and any $h_1>0$, we use $t_1,t_5,t_6$ to find $\ell_1,\ell_2,\ell_3$, and then $t_2,t_3,t_4$ to find $h_2,h_3,h_4$. 
 By defining the map $\varphi\colon \R^6_{>0} \rightarrow \R^3$ by
 \begin{align*}
\varphi_1(t)&:=t_1+t_2+t_3+t_4+t_5, \notag\\
\varphi_2(t)&:=(t_1+t_4-t_6) t_2 + (t_1+t_4) t_3 + t_1 t_4  + t_4 t_5  + t_4 t_6,     \\
\varphi_3(t)&:=t_1t_2t_4+t_1t_3t_4,\notag
\end{align*}
this gives that
\begin{align*}
\mathcal{P}:=&\{p_{(h,\l)} :  (h,\l)\in \R^4_{>0}\times \R^3_{>0}\}
=\big\{z^3+\varphi_1(t)z^2+\varphi_2(t)z+\varphi_3(t) : t\in \R^6_{>0} \big\}.
\end{align*}
Using \texttt{Mathematica} with the command \texttt{Reduce} \cite{reference.wolfram_2024_reduce}, 
we find the semi-algebraic conditions describing the image of $\varphi$ and obtain:
\begin{align*}\mathcal{P}:=&\{p_{(h,\l)} :  (h,\l)\in \R^4_{>0}\times \R^3_{>0}\}
=\big\{z^3+b_1z^2+b_2z+b_3 : b_1>0, b_3>0, b_3<\tfrac{b_1^3}{27} \big\}.
\end{align*}

Repeating the process for the reduced network obtained by removing the reaction with label $\k_5$, we obtain 
\begin{align*}
\mathcal{P}' := &\{p_{(h,\l')} :  (h,\l')\in \R^4_{>0}\times \R^{2}_{>0}\}\\
=& \mathcal{P} \cup \{
z^3+b_1z^2+b_2z+b_3 : b_1>0,  b_3=\tfrac{b_1^3}{27} \text{ and }b_2 >\tfrac{b_1^2}{3}
\}.
\end{align*}
Hence $\mathcal{P}'\setminus \mathcal{P}\neq \emptyset,$ 
showing that the inclusion \eqref{eq:polys} can be strict. However, a straightforward computation shows that the polynomials in $\mathcal{P}'\setminus \mathcal{P}$ do not satisfy the assumptions of Proposition~\ref{cor:hurwitz} and hence, do not have purely imaginary roots. We conjecture that this will always be the case, and that with this in place, Conjecture~\ref{conj:2} would follow from Conjecture~\ref{conj:3}. 
\end{example}

\begin{example}\label{network:processivedistributive}
Consider the   dual phosphorylation network where phosphorylation occurs processively, that is, one encounter of the kinase with the substrate leads to the phosphorylation of both sites, and dephosphorylation is distributive as in network \eqref{eq:2sitefullnetwork}:
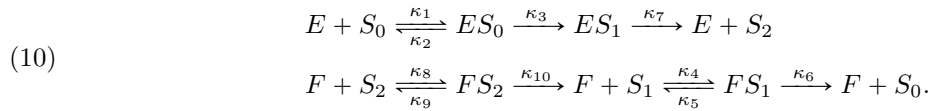
	\begin{small}
		\begin{equation}\label{eq:network5}
		\begin{aligned}
		E+ S_0  \ce{<=>[\k_{1}][\k_{2}]} &ES_0 \ce{->[\k_{3}]}   ES_1\ce{->[\k_7]} E+S_2  \\
		F+ S_{2}  \ce{<=>[\k_{8}][\k_{9}]}  &FS_{2} \ce{->[\k_{10}]} F+ S_{1}  \ce{<=>[\k_{4}][\k_{5}]} FS_1 \ce{->[\k_{6}]}F+S_0.
		\end{aligned}
		\end{equation}
	\end{small}%
 It was shown in \cite{SK} that the network admits Hopf bifurcations, and in particular, there exists $(h,\l)$ such that $p_{(h,\l)}$ has purely imaginary roots. A more detailed analysis can be found in \cite{CMS}. 

This network, with the motif $E+ S_0  \ce{<=>} ES_0 \ce{->}   ES_1$ and $X_1=E$, $X_2=S_0$, and $X_3=ES_0$, satisfies the assumptions of Theorem~\ref{thm:yes} with $\delta=0$. Hence, the reduced characteristic polynomial of the reduced network
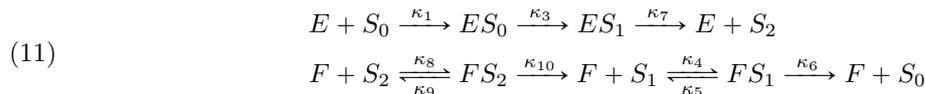
\begin{small}
	\begin{equation}\label{eq:network6}
	\begin{aligned}
	E+ S_0  \ce{->[\k_{1}]} &ES_0 \ce{->[\k_{3}]}   ES_1\ce{->[\k_7]} E+S_2  \\
	F+ S_{2}  \ce{<=>[\k_{8}][\k_{9}]}  &FS_{2} \ce{->[\k_{10}]} F+ S_{1}  \ce{<=>[\k_{4}][\k_{5}]} FS_1 \ce{->[\k_{6}]}F+S_0
	\end{aligned}
	\end{equation}
\end{small}%
also has a pair of purely imaginary roots for some  parameters and steady state. 
There are two additional instances of motifs   \eqref{motif2}. Removal of the 
reverse reactions,   with label $\kappa_9$ and $\kappa_5$, yields the network 
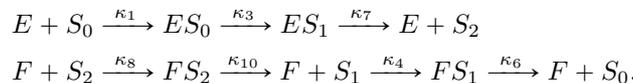
\begin{small}
	\begin{equation*}
	\begin{aligned}
	E+ S_0  \ce{->[\k_{1}]} &ES_0 \ce{->[\k_{3}]}   ES_1\ce{->[\k_7]} E+S_2  \\
	F+ S_{2}  \ce{->[\k_{8}]}  &FS_{2} \ce{->[\k_{10}]} F+ S_{1}  \ce{->[\k_{4}]} FS_1 \ce{->[\k_{6}]}F+S_0.
	\end{aligned}
	\end{equation*}
\end{small}%
For this network, the reduced characteristic polynomial also admits purely imaginary roots for suitably chosen reaction rate constants, namely
{\small	\begin{align*}
(\k_1,\k_3,\k_4,\k_6,\k_7,\k_8,\k_{10}) & = (1,1,3,40,1,50,0.9),
\end{align*}}%
and the corresponding positive steady state.
However, neither of these two motifs satisfy the assumptions of Theorem~\ref{thm:yes}. This example provides  further evidence for Conjecture~\ref{conj:2} to hold beyond the scenario of Theorem~\ref{thm:yes}.
\end{example}

\section{Dual phosphorylation networks with three intermediates}\label{sec:Hopfsimplifiednetwork}

In this section we focus on the subnetworks of the dual phosphorylation network obtained by keeping three intermediate complexes, and show that none of them admit a Hopf bifurcation. Due to the symmetry of the system pointed out in the introduction, this follows from the following theorem on the networks \ref{eq:3-intmdnetwork2} and \ref{eq:3-intmdnetwork1}. 

\begin{theorem}\label{thm:Hopf}
Hopf bifurcations do not arise in the  networks \ref{eq:3-intmdnetwork2} and \ref{eq:3-intmdnetwork1}, nor in any subnetwork obtained by making some or all reversible reactions irreversible.
\end{theorem}

 The rest of the section is devoted to proving Theorem~\ref{thm:Hopf}, first for  network~\ref{eq:3-intmdnetwork2}, and then for \ref{eq:3-intmdnetwork1}. While the analysis of \ref{eq:3-intmdnetwork1} is straightforward from Proposition~\ref{cor:hurwitz}, the analysis of \ref{eq:3-intmdnetwork2} builds on Theorem~\ref{thm:yes} and necessitates several strategic approaches to deal with the high computational cost.

\subsection{Network~\ref{eq:3-intmdnetwork2}}
In this subsection we prove Theorem~\ref{thm:Hopf} for the network~\ref{eq:3-intmdnetwork2} and its subnetworks. 
We start by noting that 
by setting  $S_2,F,$ and $FS_2$ as $X_1,X_2,$ and $X_3$ respectively and $\delta=1$, 
assumptions (A1)-(A5)  hold for \ref{eq:3-intmdnetwork2} and for the subnetwork obtained by removing the reaction $KS_1\ce{->} S_1+K$ and/or $KS_0\ce{->} S_0+K$ from \ref{eq:3-intmdnetwork2}. 
Thus, in view of Theorem~\ref{thm:yes} and Corollary~\ref{cor:Hopf}, 
the absence of Hopf bifurcations for \ref{eq:3-intmdnetwork2} and for its subnetworks follows from the 
absence of Hopf bifurcations for the following reduced network and its subnetworks:
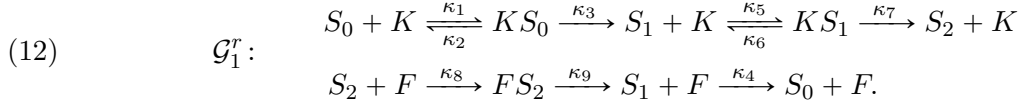
\begin{align}\label{eq:3-intmdnetwork2reduced}
\mathcal{G}_1^r \colon
\begin{split}
 \qquad S_0 + K & \ce{<=>[\k_1][\k_2]} KS_0 \ce{->[\k_3]} S_1+K   \ce{<=>[\k_5][\k_6]} KS_1 \ce{->[\k_7]} S_2+K \\
S_2 + F  & \ce{->[\k_{8}]} FS_2 \ce{->[\k_{9}]} S_1+F   \ce{->[\k_4]} S_0+F.
\end{split}
\end{align}

After ordering the species as $K, F,S_0,S_1,S_2,KS_0,FS_2,KS_1$,  
the stoichiometric,  reactant and extreme matrices are  
{\scriptsize
	\begin{align*}
	N=\begin{bmatrix}
	-1 & 1 & 1 & 0 & -1 & 1 & 1 & 0 & 0 
	\\
	0 & 0 & 0 & 0 & 0 & 0 & 0 & -1 & 1 
	\\
	-1 & 1 & 0 & 1 & 0 & 0 & 0 & 0 & 0 
	\\
	0 & 0 & 1 & -1 & -1 & 1 & 0 & 0 & 1 
	\\
	0 & 0 & 0 & 0 & 0 & 0 & 1 & -1 & 0 
	\\
	1 & -1 & -1 & 0 & 0 & 0 & 0 & 0 & 0 
	\\
	0 & 0 & 0 & 0 & 0 & 0 & 0 & 1 & -1 
	\\
	0 & 0 & 0 & 0 & 1 & -1 & -1 & 0 & 0 
	\end{bmatrix}, \ 
	B=\begin{bmatrix}
	1 & 0 & 0 & 0 & 1 & 0 & 0 & 0 & 0 
	\\
	0 & 0 & 0 & 1 & 0 & 0 & 0 & 1 & 0 
	\\
	1 & 0 & 0 & 0 & 0 & 0 & 0 & 0 & 0 
	\\
	0 & 0 & 0 & 1 & 1 & 0 & 0 & 0 & 0 
	\\
	0 & 0 & 0 & 0 & 0 & 0 & 0 & 1 & 0 
	\\
	0 & 1 & 1 & 0 & 0 & 0 & 0 & 0 & 0 
	\\
	0 & 0 & 0 & 0 & 0 & 0 & 0 & 0 & 1 
	\\
	0 & 0 & 0 & 0 & 0 & 1 & 1 & 0 & 0 
	\end{bmatrix}, \ 
	E=\begin{bmatrix}
	1&1&0&0\\
	0&1&0&0\\
	1&0&0&0\\
	1&0&0&0\\
	0&0&1&1\\
	0&0&0&1\\
	0&0&1&0\\
	0&0&1&0\\
	0&0&1&0
	\end{bmatrix}.
	\end{align*}
}

The Jacobian matrix $J(h,\l)$ from \eqref{eqn:Jac-convex}  for $(h,\l) \in \R^8_{>0} \times \R^4_{>0}$ is
{\scriptsize
	\begin{align*}
	\begin{bmatrix}
	(-\l_{1}-\l_{2}-\l_{3}-\l_{4}) h_{1} & 0 & (-\l_{1}-\l_{2}) h_{3} & (-\l_{3}-\l_{4}) h_{4} & 0 & (\l_{1}+\l_{2}) h_{6} & 0 & (\l_{3}+\l_{4}) h_{8} 
	\\
	0 & -\l_{3} h_{2} & 0 & 0 & -\l_{3} h_{5} & 0 & \l_{3} h_{7} & 0 
	\\
	(-\l_{1}-\l_{2}) h_{1} & \l_{1} h_{2} & (-\l_{1}-\l_{2}) h_{3} & \l_{1} h_{4} & 0 & \l_{2} h_{6} & 0 & 0 
	\\
	(-\l_{3}-\l_{4}) h_{1} & -\l_{1} h_{2} & 0 & (-\l_{1}-\l_{3}-\l_{4}) h_{4} & 0 & \l_{1} h_{6} & \l_{3} h_{7} & \l_{4} h_{8} 
	\\
	0 & -\l_{3} h_{2} & 0 & 0 & -\l_{3} h_{5} & 0 & 0 & \l_{3} h_{8} 
	\\
	(\l_{1}+\l_{2}) h_{1} & 0 & (\l_{1}+\l_{2}) h_{3} & 0 & 0 & (-\l_{1}-\l_{2}) h_{6} & 0 & 0 
	\\
	0 & \l_{3} h_{2} & 0 & 0 & \l_{3} h_{5} & 0 & -\l_{3} h_{7} & 0 
	\\
	(\l_{3}+\l_{4}) h_{1} & 0 & 0 & (\l_{3}+\l_{4}) h_{4} & 0 & 0 & 0 & (-\l_{3}-\l_{4}) h_{8} 
	\end{bmatrix}.
	\end{align*}
}

The two pairs of reversible reactions are represented in the second and fourth columns of $E$, and making one such pair irreversible corresponds to letting $\l_2=0$ or $\l_4=0$, as we show in Lemma~\ref{lem:EmatrixG}.
Since $\rk(N)=5$,  the reduced characteristic polynomial from \eqref{eq:charpoly} has the form
\begin{align*}
p_{(h,\l)}= z^5+a_1(h,\l)z^4+a_2(h,\l)z^3+a_3(h,\l)z^2+a_4(h,\l)z+a_5(h,\l).
\end{align*}

We compute the Hurwitz determinants  in \texttt{Maple} \cite{maple} and  find that 
for $i=1,2,3$, all coefficients of the polynomial
$\det(H_i)$   are positive, and hence it is positive after evaluated at any $(h,\l) \in \R^8_{>0} \times \R^4_{>0}$. 
However, both the polynomials $\det(H_4)$ and $a_5$ have coefficients of both signs. In view of Proposition~\ref{cor:hurwitz},  to show that the network $\mathcal{G}_1^r$ in~\eqref{eq:3-intmdnetwork2reduced} does not admit a Hopf bifurcation all we need is to  show that
if  $a_5(h,\l) >0$ for some $(h,\l) \in \R^8_{>0} \times \R^4_{>0}$, then $\det(H_4(h,\l)) \neq 0.$
In fact, we will show a stricter condition, namely that for $(h,\l) \in \R^8_{>0} \times \R^4_{>0}$,
\begin{equation}\label{eq:goal}
 a_5(h,\l) >0 \qquad \Rightarrow \qquad \det(H_4(h,\l))> 0,
\end{equation}
holds for $\mathcal{G}_1^r$ and also when $\l_2\cdot \l_4=0$. With this in place, the proof of Theorem~\ref{thm:Hopf} will be in place.

 \medskip
\noindent
\textit{\textbf{The term $a_5(h,\l)$. }} 
A computation gives that 
{\small \begin{align*}
	a_5(h,\l)&=\big((-h_1 h_2 h_3 h_4 - h_1 h_2 h_3 h_5 - h_1 h_2 h_4 h_5 - h_1 h_3 h_4 h_5 - h_2 h_3 h_4 h_5 + h_2 h_3 h_4 h_6 -  h_2 h_4 h_5 h_6  
 \\&+ h_3 h_4 h_5 h_6 - h_1 h_3 h_4 h_7 + h_1 h_4 h_5 h_7 + h_3 h_4 h_5 h_7 + h_3 h_4 h_6 h_7 + 
	h_3 h_5 h_6 h_7 + h_4 h_5 h_6 h_7) h_8 
 \\& +h_1 h_3 h_5 h_6 h_7 + 2 h_1 h_4 h_5 h_6 h_7 + h_3 h_4 h_5 h_6 h_7\big)\ \l_1 (\l_1 + \l_2) 
	\l_3^2  (\l_3 + \l_4).
	\end{align*}}
We write $a_5(h,\l)$ as
\begin{align*}
a_5(h,\l)&=(c_5(h) h_8 +b_5(h))\, \l_1 (\l_1 + \l_2) \l_3^2  (\l_3 + \l_4)
\end{align*}
where
{\small \begin{equation}\label{eq:bc}
	\begin{aligned}
	c_5(h)&:=(h_6-h_1)(h_2+h_5+h_7)h_3 h_4+(h_7-h_2)(h_1+h_3+h_6)h_4 h_5 +(h_6 h_7-h_1 h_2)h_3 h_5, \\
	b_5(h)&:= (h_1 h_3  + 2 h_1 h_4  + h_3 h_4 )h_5 h_6 h_7.
	\end{aligned}
	\end{equation}}%
In particular, $b_5(h)$ and $c_5(h)$ depend only on $h$. As the sign of $a_5(h,\ell)$ does not depend on $\l_1,\dots,\l_4>0$, we define
\[\widetilde{a}_5(h):=c_5(h) h_8 +b_5(h). \]
Then \eqref{eq:goal} follows if for $h \in \R^8_{>0}$,
\begin{equation}\label{eq:goal2}
 \widetilde{a}_5(h) >0 \qquad \Rightarrow \qquad \det(H_4(h,\l))> 0 \text{ for all }\l \in \R^4_{>0}
\end{equation}
holds.

\medskip
\noindent
\textit{\textbf{Cases. }}
We show \eqref{eq:goal2} by considering  the following six cases on how the entries of $h$ relate to each other: 

\medskip
\begin{center}
	\begin{tabular}{|l|l| c| }
		\hline
		\textbf{Cases}  & \textbf{Subcase} & \textbf{Condition on $h$}  \\
		\hline
		Case $1$ & 1a & $h_6\geq h_1,\ h_7\geq h_2$  \\ 			\cline{2-3}
		& 1b & $h_1> h_6,~~h_7\geq h_2,\ h_5\geq h_3$\\
		\cline{2-3}
		& 1c & $h_6\geq h_1,~~h_2> h_7,\ h_3\geq h_5$\\
		\hline
		Case $2$ & 2a & $h_1 > h_6,~~h_7> h_2,~~h_3 > h_5$, $c_5(h) \geq 0$  \\
		\cline{2-3}
		& 2b & $h_6> h_1,\ h_2> h_7,\ h_5 > h_3$, $c_5(h)\geq 0$  \\
		\hline
		Case $3$ & 3a & $h_1> h_6,\ h_7\geq h_2,\ h_3 > h_5$, $c_5(h)< 0$ \\ \cline{2-3}
		& 3b & $h_6\geq h_1,\ h_2> h_7,\ h_5 > h_3$, $c_5(h)< 0$ \\ 	\cline{2-3}
		& 3c & $h_1> h_6,\ h_2 >  h_7$ ($c_5(h)< 0$) \\
		\hline
	\end{tabular}
\end{center}

\smallskip

These cases cover all scenarios. In particular note that if $h_2=h_7$ and $h_1>h_6$, then $c_5(h)<0$ (similarly if $h_1=h_6$ and $h_2>h_7$).

We check \eqref{eq:goal2} using \texttt{Maple} and encode the inequalities describing each case using auxiliary variables. That is, each inequality 
$h_i> h_j$  is encoded by making the substitution $h_i=h_j+v_i$ into $\det(H_4)$. If the resulting polynomial  only has positive coefficients, then $\det(H_4)$ is positive under the tested case. The case $h_i=h_j$ is treated separately. 

In some cases below, we will have to impose that a certain variable, say $v$, is smaller or larger than a rational function $q$  in the rest of the variables and with positive denominator. This will be implemented with the substitutions
\[ v= \tfrac{\mu}{\mu+1}\, q , \ \text{for }v<q, \qquad\text{and}\qquad v= q+\mu, \ \text{for }v> q,\]
where $\mu$ is then assumed to take positive values. We will  consider the numerator of $\det(H_4)$ after performing these substitutions and check that all coefficients are positive.

\medskip
\noindent
\textit{\textbf{Breaking down the problem. }}
The implication \eqref{eq:goal2}  follows from the following proposition, as the subnetworks obtained by making some reversible reactions irreversible, correspond to setting $\l_2=0$ and/or $\l_4=0$ in $\det(H_4)$ and in $a_5(h,\l)$. 

 \begin{proposition}\label{prop:pos}
 After viewing  $\det(H_4)$  as a polynomial in $\l_1,\l_2,\l_3,\l_4$, each coefficient is a polynomial in $h_1,\ldots,h_8$ that only attains positive values when evaluated at any  $h\in \R^8_{>0}$ for which 
$\widetilde{a}_5(h) >0$. 
Moreover, this remains true after setting $\l_2$ and/or $\l_4$ to zero.
\end{proposition}

We explain here the rationale and computational approach taken in the accompanying \texttt{Maple} file to show the claim in Proposition~\ref{prop:pos} for all the cases in the table. 
First, we notice that the second part follows from the first if one guarantees that $\det(H_4)$ does not vanish identically in any of the cases when $\l_2\cdot \l_4=0$. This is done by showing that the coefficient  of $l_1^7l_3^3$ in $\det(H_4)$ is positive when $h_5\geq h_3$, and that the coefficient of $l_1^9l_3$ is positive when $h_3\geq h_5$. Hence, for all cases, one coefficient is positive even if $\l_2\cdot \l_4=0$, and hence $\det(H_4)\neq 0$ as a polynomial.

\medskip
\noindent
\textit{\textbf{Case 1.  }}
In the  \textit{Subcases 1a, 1b, 1c}, we perform the substitutions by introducing two or three new variables. For  \textit{Subcase 1a}, we let
$h_6=h_1+v_1$ and $h_7=h_2+v_2$. With this, every coefficient of $\det(H_4)$ in the $\l$'s, as a polynomial in the $h$'s and $v$'s, only has positive coefficients. The same is true if $v_1=v_2=0$, and hence it remains true if only one of $v_1,v_2$ is nonzero.

Similarly, for  \textit{Subcase 1b}, we make the substitution  $h_1=h_6+v_1$, $h_7=h_2+v_2$, and $h_5=h_3+v_3$; for  \textit{Subcase 1c}, 
we make the substitution  $h_6=h_1+v_1$, $h_2=h_7+v_2$, and $h_3=h_5+v_3$. In both cases, the coefficients of $\det(H_4)$ in the $\l$'s  are polynomials in the $h$'s and $v$'s with all coefficients positive. The same holds if $v_1=v_2=v_3=0$, covering all possibilities for \textit{Subcases 1b} and  \textit{1c}. 

These computations give that Proposition \ref{prop:pos} holds under the additional restrictions of Case 1. 
For this case, we do not need to consider the sign of $\widetilde{a}_5(h)$, as $\det(H_4)$ is positive independently of that. This will not occur in the remaining cases.

\medskip
\noindent
\textit{\textbf{Case 2. }}
In this situation, in addition to the inequalities in the $h$'s, we need to impose that $c_5(h)\geq 0$.
In \textit{Subcase 2a}, we make the substitution $h_1=h_6+v_1$, $h_7=h_2+v_2$, and $h_3=h_5+v_3$, and it turns out that the coefficients of  $\det(H_4)$ are not polynomials with all coefficients positive.

We perform the substitution into $c_5(h)$ and obtain a linear polynomial in $v_1$ with negative leading coefficient: 
\begin{equation}\label{eq:v1}
\begin{aligned}
-\big( (h_4  h_5+2  h_2  h_4 + h_2  h_5)(h_5+  v_3)   +  h_4  v_2 v_3\big) v_1
+  h_5  v_2 (   2  h_4  h_6 +  (h_4+h_6) (h_5+ v_3) ).
\end{aligned}
\end{equation}
Therefore, for $c_5(h)\geq 0$ in  \textit{Subcase 2a},  $v_1$ needs to be smaller than the only root of this polynomial:
\begin{equation}\label{eq:s1}
s_1:=\frac{h_5  v_2 (   2  h_4  h_6 +  (h_4+h_6) (h_5+ v_3) )}{(h_4  h_5+2  h_2  h_4 + h_2  h_5)(h_5+  v_3)   +  h_4  v_2 v_3 }.
\end{equation}

We could now impose  $v_1\leq  s_1$ in the coefficients of $\det(H_4)$, but the computational cost is minimized if we proceed as follows. 
We make first only the substitution $h_1=h_6+v_1$ in  $p_{(h,\ell)}$, and observe that all coefficients become linear in $v_1$. 
We compute the generic form of the Hurwitz determinant $\det(H_4)$ for a polynomial of degree five $z^5+\alpha_1z^{4}+\ldots+\alpha_5$ and make the substitutions
\[\alpha_i= c_i v_1 + b_i,\quad i=1,\dots,4,\qquad \alpha_5 = (u_0\, v_1 + u_1) h_8 + (u_2\, v_1+u_3), \]
where $c_i,b_i,u_i$ are unspecified parameters.
This gives rise to a polynomial $q$ in $c_1,\dots,c_4$, $b_1,\dots,b_4,u_0,\dots,u_3,v_1,h_8$. 
In this polynomial, we perform the substitution
\[v_1 = - \frac{\mu}{\mu+1} \frac{u_1}{u_0}, \qquad \mu>0,\]
to impose first that $v_1<s_1$. 
We obtain a rational function in $\mu$, with positive denominator and numerator of degree $4$. The coefficient of $\mu^i$ has the factor $u_0^{4-i}$. 
As $u_0<0$ in the situation under study (it is the coefficient of $v_1$ in \eqref{eq:v1}) and we want to verify that $q$ is positive, we 
consider 
\[ M_i = \frac{(-1)^i}{u_0^{4-i}} (\text{coefficient of }\mu^i\text{ in } q). \]
When $v_1=s_1$ (to impose $c_5(h)=0$), the resulting polynomial is $M_4$. 

All we need at this point is to verify $M_i> 0$ when we 
consider the actual expressions of $c_0,\dots,c_4,b_0,\dots,b_4,u_0,\dots,u_3$ from   $p_{(h,\ell)}$, and additionally make the substitutions $h_7=h_2+v_2$ and $h_3=h_5+v_3$. 
After all these substitutions, for each $M_i$, we collect the coefficients of $M_i$ in $\l_1,\l_2,\l_3,\l_4$ and verify that they are all positive. To avoid memory issues, we view first each coefficient as a polynomial in $v_2,v_3$ and then verify that its coefficients are positive. 

The computations in the accompanying file show that the coefficients are indeed positive.

\smallskip
The \textit{Subcase 2b} is analogous. We make the substitution $h_6=h_1+v_1$, $h_2=h_7+v_2$, and $h_5=h_3+v_3$. The coefficients of $\det(H_4)$ are not all positive. 
Performing the substitution into $c_5(h)$ yields the following polynomial:
\begin{equation*}
\begin{aligned}
- \big( (h_1h_3  +2h_1h_4 + h_3h_4)(h_3+v_3)   +  h_4v_1 v_3\big)v_2 + h_3v_1((h_3+v_3)(h_4 + h_7) + 2h_4h_7 ).\end{aligned}
\end{equation*}
The polynomial is linear in $v_2$, has negative leading term, and positive constant term. So $c_5(h)\geq 0$  if $v_2$ is smaller than
\begin{equation}\label{eq:s2}
s_2:=\frac{h_3v_1((h_3+v_3)(h_4 + h_7) + 2h_4h_7)}{ ( h_1h_3  +2h_1h_4 + h_3h_4)(h_3+v_3)   +  h_4v_1 v_3 }.
\end{equation}
We proceed identically as in  \textit{Subcase 2a}, but with the role of $v_1$  taken by $v_2$ now. 
We reduce the problem of checking that five coefficients $M_0,\dots,M_4$ are positive, and this is verified in the accompanying file.

\smallskip
\noindent
\textit{\textbf{Case 3. }}
The three subcases have $c_5(h)<0$. 
As $b_5(h)>0$, we require $h_8<-\tfrac{b_5(h)}{c_5(h)}$ for $\widetilde{a}_5(h)>0$. We proceed similarly to Case 2. We view first the coefficients of $p_{(h,\ell)}$ as linear polynomials in $h_8$, by expressing them as $c_i h_8 + b_i$ for $i=0,\dots,5$ and with $b_i,c_i$ treated as undetermined parameters. (We will later substitute $c_5$ by $c_5(h)$ and $b_5$ by $b_5(h)$ from \eqref{eq:bc}.)

We perform this substitution into a generic Hurwitz determinant $\det(H_4)$ to obtain a polynomial $q$
in $c_0,\dots,c_5,b_0,\dots,b_5,h_8$. 
With this in place, we perform the substitution 
\[h_8 = - \frac{\mu}{\mu+1} \frac{b_5}{c_5},\]
into $q$, to impose that $h_8<-\frac{b_5}{c_5}$. 
We obtain a rational function with positive denominator and numerator of degree $4$ in $\mu$. As in Case 2, $c_5^{4-i}$ is a multiple of the coefficient of $\mu^i$, for $i=0,\dots,3$. For $i=4$, the coefficient has the factor $-(b_4c_5-b_5c_4)$. 
So, as $c_5<0$ in the current situation, we define
\begin{align*}
M_i &= \frac{(-1)^i}{c_5^{4-i}} (\text{coefficient of }\mu^i\text{ in }q), \quad i=0,\dots,3,\\
M_4 &= \frac{1}{M_5} (\text{coefficient of }\mu^4\text{ in }q), \\
M_5 &= -(b_4c_5-b_5c_4).
\end{align*}
We substitute $c_i,b_i$ by their expressions from $p_{(h,\ell)}$ and it turns out that $M_0$ is directly positive. 
For $M_1,\dots,M_5$, we consider each subcase separately. 

For \textit{Subcase 3c}, we perform the substitution $h_1 = h_6 + v_1, h_2 = h_7 + v_2$. 
For  \textit{Subcase 3b}, we need $v_2>s_2$ with $s_2$ as in \eqref{eq:s2} to guarantee that $c_5(h)<0$, so we perform the substitution 
$h_6 = h_1 + v_1, h_2 = h_7 +s_2+ v_4,h_5=h_3+v_3$. Similarly, for  \textit{Subcase 3a} we perform the substitution 
$h_1 = h_6 + s_1+v_4, h_7 = h_2 + v_2,h_3=h_5+v_3$ with $s_1$ as in \eqref{eq:s1}. 
In all subcases, we obtain that all coefficients of the $M_i$'s in the $\l$'s have positive coefficients.  \textit{Subcase 3a} is computationally the most expensive. 

As $M_0$ is not zero, the boundary cases $h_2=h_7$ or $h_1=h_6$ follow as well.

\smallskip
With these computations in place, Proposition~\ref{prop:pos} and hence also  Theorem~\ref{thm:Hopf} for network  \ref{eq:3-intmdnetwork2} hold. The whole computation for the proof of Proposition~\ref{prop:pos} took about $9.4$ hours in a Macbook Pro with 32 GB RAM and chip Apple M1 Pro.

\subsection{Network  \ref{eq:3-intmdnetwork1}}\label{sec:remainingnetworks}\label{sec:network2}
We now prove Theorem~\ref{thm:Hopf} for network  \ref{eq:3-intmdnetwork1}:
\begin{align*}
	\begin{split}
	S_0 + K \ce{<=>[\k_1][\k_2]} KS_0 \ce{->[\k_3]} S_1+K  \ce{<=>[\k_7][\k_8]} KS_1 \ce{->[\k_9]} S_2+K \\
	S_2 + F   \ce{->[\k_{10}]} S_1+F  \ce{<=>[\k_4][\k_5]} FS_1 \ce{->[\k_6]} S_0+F.
	\end{split}
	\end{align*}

With the order of species $K,F,S_0,S_1,S_2,KS_0,FS_1,KS_1$,  the stoichiometric, reactant, and   extreme matrices are
{\scriptsize
	\begin{align*}
	N&=\begin{bNiceArray}{cccccccccc}
	-1&1&1&0&0&0&-1&1&1&0 \\
	0&0&0&-1&1&1&0&0&0&0\\ 
	-1&1&0&0&0&1&0&0&0&0\\
	0&0&1&-1&1&0&-1&1&0&1\\
	0&0&0&0&0&0&0&0&1&-1\\
	1&-1&-1&0&0&0&0&0&0&0\\
	0&0&0&1&-1&-1&0&0&0&0\\
	0&0&0&0&0&0&1&-1&-1&0
	\end{bNiceArray},\quad
	B=\begin{bNiceArray}{cccccccccc}
	1&0&0&0&0&0&1&0&0&0 \\
	0&0&0&1&0&0&0&0&0&1\\ 
	1&0&0&0&0&0&0&0&0&0\\
	0&0&0&1&0&0&1&0&0&0\\
	0&0&0&0&0&0&0&0&0&1\\
	0&1&1&0&0&0&0&0&0&0\\
	0&0&0&0&1&1&0&0&0&0\\
	0&0&0&0&0&0&0&1&1&0
	\end{bNiceArray}, 
 \end{align*}
 \begin{align*}
	E^{\top}=\begin{bNiceArray}{cccccccccc}
 1&0&1&1&0&1&0&0&0&0\\
 1&1&0&0&0&0&0&0&0&0\\
 0&0&0&0&0&0&1&0&1&1\\
 0&0&0&1&1&0&0&0&0&0\\
 0&0&0&0&0&0&1&1&0&0
 \end{bNiceArray}.
\end{align*}
}

As with \ref{eq:3-intmdnetwork2}, we have that  $p_{(h,\ell)}$ has degree $s=5$. 
In this case, for $i=1,\dots,4$, the Hurwitz determinants $\det(H_i(h,\l))$  are polynomials in $h,\ell$ with only positive coefficients. The same remains true when setting $\l_2,\l_4$ and $\l_5$ to zero, which are the parameters corresponding to the pairs of reversible reactions in \ref{eq:3-intmdnetwork1}. 
See the accompanying \texttt{Maple} file for details. 

Hence, Proposition~\ref{cor:hurwitz}   readily gives Theorem~\ref{thm:Hopf} for  \ref{eq:3-intmdnetwork1}.

\section{Proof of Theorem~\ref{thm:yes}}\label{sec:proof}

In this section we prove Theorem~\ref{thm:yes}. We do it in several steps, to illustrate where the assumptions (A1)-(A5) of 
Theorem~\ref{thm:yes} play a role. 

Given a vector $v=(v_1,\ldots,v_k)\in \R^k$ and $1\leq i<j\leq k$, we denote by $v_{i:j}$ the vector of size $j-i+1$ given by $(v_i,v_{i+1},\ldots, v_{j-1}, v_j)$. We let $\mathbf{0}_{ i \times j}$ denote the zero matrix of size $i\times j$.

Let $G$ be a network containing motif~\eqref{motif2}. Assuming $G$ has $r$ reactions, we order the set of reactions such that the three reactions of the motif are last and ordered as indicated by the following labels: 
 \begin{equation}\label{motif}
 y \ce{<=>[\k_{r-1}][\k_r]} Y \ce{->[\k_{r-2}]} y'.
\end{equation}
Let $G'$ be the network obtained by removing the reaction with label $\k_r$, and with the same order of species and reactions as $G$. In order to prove Theorem~\ref{thm:yes}, we follow these steps:
\begin{itemize}
\item Step 1: Understand how the Jacobian matrices, $J(h,\ell)$ for $G$ and $J'(h,\ell')$  for $G'$, relate to each other, assuming only (A2), that is, that species $Y$ only appears in the motif. 
\item Step 2: Obtain the specific form of $J(h,\ell)$ under all assumptions (A1)-(A5). 
\item Step 3: 
Compute relevant parts of the reduced characteristic polynomial of $J(h,\ell)$ and conclude Theorem~\ref{thm:yes}. 
\end{itemize}

\medskip
\noindent
\textbf{Step 1. Jacobian matrices of $G$ and $G'$.}
We let $N\in \Z^{n\times r}$ and $N'\in \Z^{n\times (r-1)}$ denote the stoichiometric matrices of $G$ and $G'$, and likewise, 
$B\in \Z^{n\times r}$ and $B'\in \Z^{n\times (r-1)}$ denote their reactant matrices. 
Let $N''\in \Z^{n\times (r-3)}$ and $B''\in \Z^{n\times (r-3)}$ be the submatrices of $N$ and $B$ obtained by keeping the first $r-3$ reactions, and let $\rho_2, \rho_1\in \R^n$ be respectively the column vectors corresponding to the reactions $Y \ce{->} y'$ and   $y \ce{->} Y$.  By the choice of ordering of the reactions of $G$ and $G'$, we then have:
\begin{align} \label{eq:Nblock}
N &= \begin{bNiceArray}{c|ccc} 	N'' &  \rho_2 & \rho_1 & -\rho_1 	\end{bNiceArray}, 
& N' &=  \begin{bNiceArray}{c|cc} 	N'' &  \rho_2 & \rho_1  	\end{bNiceArray}.
\end{align}
Similarly, by letting $e_Y\in \R^n$ be the column vector with all  entries equal to zero except for the entry corresponding to $Y$, which is equal to $1$, and $e_y\in \R^n$ be the column vector with the stoichiometric coefficients of $y$, we have
\begin{align} \label{eq:Bblock}
B &= \begin{bNiceArray}{c|ccc} 	B'' &  e_Y & e_y  & e_Y 	\end{bNiceArray}
& B' &=  \begin{bNiceArray}{c|cc} 	B'' &  e_Y & e_y   	\end{bNiceArray}.
\end{align}

\begin{lemma}\label{lem:EmatrixG}
Let $G$ be a network that contains motif~\eqref{motif} and such that $Y$ does not appear in any other reaction outside the motif. Let $E'\in \R^{(r-1)\times (m-1)}$ be the extreme matrix for the network $G'$. With the conventions above, the extreme matrix for network $G$ is
\begin{align*}
	E= \begin{bNiceArray}{c|c}
	E' & \Block{2-1}{E_m} \\ 
	 \mathbf{0}_{ 1 \times (m-1)} & 
	\end{bNiceArray}  \in \R^{r\times m}, 
\end{align*}
where $E_m^{\top}:=(0,\ldots,0,1,1)\in\R^{r}$.  
\end{lemma}

\begin{proof}  
	Since $N'$ and $N$ differ only in the last column,  $\omega\in \ker(N') \cap \mathbb{R}^{r-1}_{\geq0}$ if and only if $(\omega,0)\in \ker(N) \cap \mathbb{R}^{r}_{\geq0}$. It follows that the first $m-1$ columns of $E$ are extreme vectors. 
 Furthermore,   $E_m\in\ker(N) \cap \mathbb{R}^{r}_{\geq0}$ by construction (see also \cite[Lemma 8]{TF}).
Let   $\omega\in \ker(N) \cap \mathbb{R}^{r}_{\geq0}$ be an extreme vector such that  $\omega_r\neq 0$. All that is left is to show that $\omega$ is a multiple of $E_m$. 
By considering the row of $N$ corresponding to species $Y$, it follows that 
\[-\omega_{r-2}+\omega_{r-1}-\omega_{r}=0 \qquad \Rightarrow \qquad   \omega_{r-1}= \omega_{r-2}+ \omega_{r}\geq  \omega_{r} \neq 0. \] 
Therefore the vector $v:=\omega-\omega_r E_m$ belongs to $\mathbb{R}^{r}_{\geq0}$  and its support is strictly contained in the support of  $\omega$. Since $v\in \ker(N)$ and $w$ is an extreme vector, $v=0$ and we conclude that $\omega$ is a multiple of $E_m$ as desired.
\end{proof}
 
 We define the following $n\times n$ matrices: 
\begin{align*}
J''(L)  & := N'' \diag(L) (B'')^\top & \text{for} &  L\in \R^{r-3}, \\
J_1(\alpha,\beta) &:= \alpha\, \rho_2 e_Y^\top + \beta\, \rho_1 e_y^\top & \text{for } & \alpha,\beta \in \R, \\
J_2(\alpha) &:= \alpha \, \rho_1 (e_y -e_Y)^\top  & \text{for } & \alpha \in \R. 
\end{align*}
We define also the following maps:
\begin{align*}
\pi \colon & \R^m   \rightarrow \R^{m-1} \qquad \ell=(\ell_1,\dots,\ell_m) \mapsto \pi(\ell)=(\ell_1,\dots,\ell_{m-1}), \\
\iota \colon & \R^{m-1}   \rightarrow \R^{m} \qquad \ell'=(\ell_1,\dots,\ell_{m-1}) \mapsto \iota(\ell')=(\ell_1,\dots,\ell_{m-1},0).
\end{align*}

\begin{proposition}\label{prop:jac}
With the notation above, 
let  $J(h,\ell)$ be the Jacobian matrix  of $G$ in convex parameters $(h,\ell) \in \R^n_{>0},\times \R^m_{>0}$, and let $J'(h,\ell')$ be the Jacobian matrix  of $G'$ in convex parameters $(h,\ell') \in \R^n_{>0}\times \R^{m-1}_{>0}$. It holds
\begin{align*}
J(h,\ell) &:= J'(h,\pi(\ell)) +  J_2(\ell_{m})\diag(h) ,
\\ 
J'(h,\ell') &:= \big( J''( L_{1:r-3} )  + J_1(L_{r-2},L_{r-1})       \big) \diag(h) , \qquad \text{where} \quad L:= E'  \ell'. 
  \end{align*}
In particular,
\[ J'(h,\ell') = J(h,\iota(\ell')). \]  
\end{proposition}
\begin{proof}
The result follows readily by using the block structures of the matrices defining $J(h,\ell)$. 
As $J(h,\l)=N \diag(E\l)B^{\top} \diag(h)$, it is enough to compute $N \diag(E\l)B^{\top}$. By Lemma~\ref{lem:EmatrixG}:
\[(E\ell)_i = \begin{cases}   (E' \pi(\ell) )_i & i=1,\dots,r-2,  \\  (E' \pi(\ell) )_{r-1} + \ell_m & i=r-1, \\ \ell_m & i=r.
\end{cases}  \]
Therefore, $\diag(E\ell) = \diag( \iota(E'\pi(\ell)) ) + \diag( \ell_m E_m)$, with $E_m^\top=(0,\dots,0,1,1)$ as in  Lemma~\ref{lem:EmatrixG}. 
Hence
\[N \diag(E\l)B^{\top} = N  \diag(\iota(E'\pi(\ell))) B^\top 
+  N  \diag( \ell_m E_m) B^\top. 
\] 
Now, using the block structures of $N$ and $B$ in \eqref{eq:Nblock} and \eqref{eq:Bblock}, we have
\begin{align*}
N  \diag( \ell_m E_m) B^\top &=   \Big[ N'' \ |\ \rho_2 \ \rho_1 \ - \rho_1 \Big]   \diag( \ell_m E_m) \Big[ B''\ |\ e_Y \ e_y \ e_Y \Big]^\top \\ &= 
\ell_m \Big[  \textbf{0}_{n\times (r-2)} \ |  \  \rho_1  \ \  -   \rho_1 \Big]    \Big[ B''\ |\ e_Y \ e_y \ e_Y \Big]^\top = \ell_m  \rho_1( e_y -   e_Y )^\top= J_2(\ell_m). 
\end{align*}
Moreover, the following equality is easy to verify  
\begin{align*}
N  \diag(\iota(E'\pi(\ell))) B^\top  &=   N'  \diag(E'\pi(\ell)) (B')^\top. 
\end{align*}
This shows the first part, as $J'(h,\pi(\ell))=N'  \diag(E'\pi(\ell)) (B')^\top \diag(h)$.
 To find the structure of $J'(h,\pi(\ell))$, it is enough to show that $ N'  \diag(L) (B')^\top= J''( L_{1:r-3} )  + J_1(L_{r-2},L_{r-1}) $ for any $L\in \R^{r-1}$. 
We have
\begin{multline*}
N'  \diag(L) (B')^\top  = \Big[ N'' \ |\ \rho_2 \ \rho_1 \Big]\diag(L)   \Big[ B''\ |\ e_Y \ e_y  \Big]^\top \\ 
= N''\diag(L_{1:r-3}) (B'')^\top + L_{r-2} \rho_2 e_Y^\top + L_{r-1} \rho_1 e_y^\top 
=  J''( L_{1:r-3} )  + J_1(L_{r-2},L_{r-1})
\end{multline*}
as desired. 
\end{proof}

\bigskip
\noindent
\textbf{Step 2. Jacobian matrices for networks satisfying assumptions (A1)-(A5).}
In the previous computation, we wrote the Jacobian matrix of $G$ in terms of that of $G'$ assuming only (A2). If (A1) and (A3) hold,  the complexes $y,y'$ of motif \eqref{motif} are $y=X_1+X_2$ and $y'=c+\delta X_2$, with $c$ not depending on $X_1,X_2$. 
We are implicitly ordering the species $X_1,X_2,X_3$ as the first three species. 
For $e_{c}$ the column vector of stoichiometric coefficients of $c$ in the species $X_4,\dots,X_n$, we have that
\begin{equation}\label{eq:special}
\begin{aligned}
\rho_2 & =(0,\delta,-1, e_{c}^\top)^\top,& \rho_1 &= (-1,-1,1,0,\dots,0)^\top, \\
 e_{Y} & = (0,0,1,0,\dots,0)^\top, &  e_y & = (1,1,0,\dots,0)^\top.
 \end{aligned}
 \end{equation}
With this in place, we obtain the following result. 

\begin{proposition}\label{Jacblock}
Assume $G$ satisfies assumptions (A1)-(A5) of Theorem~\ref{thm:yes}. With the notation above, 
let  $n_1^\top,n_2^\top\in \R^{r-3}$ denote the first two rows of $N''\in \R^{n\times (r-3)}$, and let $N'''\in \R^{(n-3) \times (r-3)}$ denote the submatrix of $N''$ consisting of the last $r-3$ rows. Similarly, let $b_2^\top$ denote the second row of $B''\in \R^{n\times (r-3)}$ 
 and $B'''\in \R^{(n-3) \times (r-3)}$ denote the submatrix of $B''$ consisting of the last $r-3$ rows. 

Then,  the Jacobian matrix $J(h,\ell)$   of $G$ in convex parameters $h\in \R^n_{>0}, \ell\in \R^m_{>0}$ is of the form
\[\begin{bNiceArray}{ccc|c}
- (\ell_m+L_{r-1}) & - (\ell_m+L_{r-1}) & \ell_m   &  n_1^\top \diag(L_{1:r-3}) (B''')^\top  \\ 
- (\ell_m+L_{r-1}) & -(\ell_m+L_{r-1}) & \ell_m +  \delta L_{r-2}   & (1-\delta) n_2^\top \diag(L_{1:r-3}) (B''')^\top    \\
\ell_m+L_{r-1} &  \ell_m+L_{r-1} & -(\ell_m +L_{r-2})  &  \mathbf{0}_{1\times (n-3)}    \\ \hline
  \mathbf{0}_{ (n-3) \times 1} &  \delta  N''' \diag(L_{1:r-3}) b_2 &  L_{r-2} e_{c}   & K
\end{bNiceArray} \diag(h),  \] 
 with $L:=E' \pi(\ell)$ and $K\in \R^{(n-3)\times (n-3)}$ not depending on $\ell_m$. 
 \end{proposition}
\begin{proof} 
We find the three summands of the description of $J(h,\ell)$ in  Proposition~\ref{prop:jac}. 
Considering the vectors $\rho_1,\rho_2,e_y,e_Y$ in \eqref{eq:special}, we have from definition that
\begin{equation}\label{eq:J1}
J_1(L_{r-2},L_{r-1}) =L_{r-2}\, \rho_2 e_Y^\top + L_{r-1}\, \rho_1 e_y^\top = 
\begin{bNiceArray}{ccc|c}
-L_{r-1} & -L_{r-1} & 0 & \textbf{0}_{1\times (n-3)}   \\
-L_{r-1} & -L_{r-1} & \delta L_{r-2} & \textbf{0}_{1\times (n-3)}   \\
L_{r-1} & L_{r-1} & -L_{r-2}  & \textbf{0}_{1\times (n-3)}    \\ \hline
\textbf{0}_{ (n-3) \times 1} & \textbf{0}_{ (n-3) \times 1} & L_{r-2} e_{c} & \textbf{0}_{ (n-3) \times (n-3)}
\end{bNiceArray} .
\end{equation}
Similarly,
\begin{equation}\label{eq:J2}
J_2(\ell_m) =\ell_m \, \rho_1 (e_y^\top -  e_Y^\top)  = 
 \begin{bNiceArray}{ccc|c}
- \ell_m & - \ell_m & \ell_m & \textbf{0}_{1\times (n-3)}    \\
- \ell_m & - \ell_m & \ell_m & \textbf{0}_{1\times (n-3)}   \\
 \ell_m &  \ell_m & -\ell_m & \textbf{0}_{1\times (n-3)}    \\ \hline
 & \textbf{0}_{ (n-3) \times 3} &  & \textbf{0}_{ (n-3) \times (n-3)}
\end{bNiceArray} .
\end{equation}

The third row of $N''$ and $B''$ are zero by (A2), and the first  row of $B''$ is zero by (A4). Therefore
\begin{align}
J''(L_{1:r-3}) & = N'' \diag(L_{1:r-3}) (B'')^\top =
 \begin{bNiceArray}{c}
n_1^\top \\
n_2^\top \\
\textbf{0}_{ (r-3) \times 1} \\ \hline
N''' 
\end{bNiceArray}   \diag(L_{1:r-3})   \begin{bNiceArray}{c|c|c|c}
\textbf{0}_{ 1\times (r-3) } & b_2 & \textbf{0}_{ 1\times (r-3) }  & B''' 
\end{bNiceArray} \nonumber \\
& =
\begin{bNiceArray}{ccc|c}
0 & n_1^\top \diag(L_{1:r-3}) b_2    & 0 &  n_1^\top \diag(L_{1:r-3}) (B''')^\top  \\ 
0 & n_2^\top \diag(L_{1:r-3}) b_2 & 0   & n_2^\top \diag(L_{1:r-3}) (B''')^\top    \\
0 & 0 & 0   &  \textbf{0}_{1\times (n-3)}    \\ \hline
  \textbf{0}_{ (n-3) \times 1} &    N''' \diag(L_{1:r-3}) b_2  &    \textbf{0}_{(n-3) \times 1}  & K
\end{bNiceArray} , \label{eq:Jpp}
\end{align}
for some matrix $K\in \R^{(n-3)\times (n-3)}$ not depending on $\ell_m$.

If $\delta=0$, then (A5) gives  that $b_2=0$, and in particular, the second column of $J''(L_{1:r-3})$ also is zero.
If $\delta=1$, then (A5)  gives:
\begin{itemize}
\item If $b_{2i}\neq 0$, then $n_{1i}=0$ and hence $n_1^\top \diag(L_{1:r-3}) b_2=0$.
\item  $n_2^\top=0$, and hence the second row of $J''(L_{1:r-3})$  also is zero. 
\end{itemize}

We conclude that for all $\delta=0,1$, the submatrix of $J''(L_{1:r-3})$ with row and column indices in $\{1,2,3\}$ is the zero matrix, and further 
{\small \[ N''' \diag(L_{1:r-3}) b_2=\delta   N''' \diag(L_{1:r-3}) b_2, \quad n_2^\top \diag(L_{1:r-3}) (B''')^\top=(1-\delta) n_2^\top \diag(L_{1:r-3}) (B''')^\top. \] }%
By adding \eqref{eq:Jpp}, \eqref{eq:J1} and \eqref{eq:J2} and using Proposition~\ref{prop:jac}, we obtain the statement. 
\end{proof}

\medskip
\noindent
\textbf{Step 3. Compare characteristic polynomials. } With Proposition~\ref{Jacblock} in place, 
we proceed to compute relevant coefficients of $p_{(h,\l)}\in \R[z]$. If $s=\rk(N)$ is the degree of $p_{(h,\l)}$, then the coefficient of $z^{k}$ in $p_{(h,\l)}$ 
is given by the sum of determinants of principal minors of size $s-k$ of $J(h,\ell)$, and hence 
is a polynomial in $h,\l$.  
We decompose $p_{(h,\ell)}$  as the sum of two polynomials as
 \begin{align}\label{eq:psplitting}
  p_{(h,\l)}&:=\alpha_{(h,\l)}+\beta_{(h_{4:n},\l)} \ \in \R[z],
 \end{align}
 such that $\beta_{(h_{4:n},\l)}$ is the part not depending on $h_1,h_2,h_3$; formally, 
 $\beta_{(h_{4:n},\l)} = p_{(\widetilde{h},\l)}$ with $\widetilde{h}=(0,0,0,h_4,\dots,h_n)$. 
 We write $\alpha_{(h,\l)}$ as
\begin{align}\label{eq:alpha}
	&\alpha_{(h,\ell)}:=z^s+d_1z^{s-1}+d_2z^{s-2}+\ldots+d_{s-1}z+d_s,
\end{align}
where in each $d_i$ we omit explicit reference to $(h,\ell)$ for simplicity.

\begin{proposition}\label{prop:charpolycoeffG}
Recall the description of the matrix $J(h,\ell)$ from Proposition~\ref{Jacblock}. 
	Let $T_i$ be the sum of the principal minors of $K$ of size $i$ for $i=1,\ldots,s-2$, and let as usual $L:= E' \pi(\ell)\in \R^{r-1}$. The coefficients of the polynomial $\alpha_{(h,\ell)}$   in \eqref{eq:alpha} are given by:
\begin{align*}
d_1 & = - (\l_m+L_{r-1}) h_1 - (\l_m+L_{r-1})h_2 - (\l_m+L_{r-2})h_3,  \\
d_2 & = d_1T_1 - (\l_m+L_{r-1})L_{r-2}h_1h_3+(1-\delta)(\l_m+L_{r-1})L_{r-2}h_2h_3, \\
d_{k} & = 
 d_1 T_{k-1}  +  (\l_m+L_{r-1}) \big(  \delta h_1h_2   S_k^{(1)} + h_1h_3   S_k^{(2)}  + (1-\delta) h_2h_3   S_k^{(3)}\big)
 \qquad  \text{for }\  3\leq k\leq~s,
	\end{align*}
where 
$S_k^{(1)}$, $S_k^{(2)}$, $S_k^{(3)}$ are polynomials in $h_4,\dots,h_n,\l_1,\dots,\ell_{m-1}$.

Furthermore, the polynomial $\beta_{(h_{4:n},\l)}$  defined in \eqref{eq:psplitting} does not depend on $\ell_m$. 
\end{proposition}

\begin{proof}
Every coefficient of $\beta_{(h_{4:n},\l)}$ does not involve $h_1,h_2,h_3$. Hence, these coefficients arise from the minors not involving rows and columns of $J(h,\ell)$ with index  in $\{1,2,3\}$, that is, minors of $K$. As $K$ is independent of $\ell_m$ by Proposition~\ref{Jacblock}, the last claim follows. 

We denote the submatrix of $J(h,\ell)$ obtained from rows and columns with indices $i_1<\ldots < i_k$ as $M_{\{i_1,\ldots,i_k\}}$, and the corresponding principal minor as  $m_{\{i_1,\ldots,i_k\}}$. 
Since every   coefficient of $\alpha_{(h,\l)}$ has at least one among $h_1,h_2,$ and $h_3$, the coefficient $d_k$ is   the sum of the principal minors of size $k$ obtained from submatrices $M_{\{i_1,\ldots,i_k\}}$ that include at least one of the first three columns of $J(h,\ell)$. In other words, $d_k=\sum_{i_1<\ldots <i_k}m_{\{i_1,\ldots,i_k\}}$ where the sum is over sets such that $\{i_1,\ldots,i_k\}\cap \{1,2,3\}\neq \emptyset$.
	To compute $d_1$, we consider the principal minors of size 1 given by the diagonal entries of $J(h,\l)$ shown in Proposition~\ref{Jacblock}: 
	 \[d_1=m_{\{1\}}+m_{\{2\}}+m_{\{3\}}=- (\l_m+L_{r-1}) h_1 - (\l_m+L_{r-1})h_2 - (\l_m+L_{r-2})h_3.\]
	 
To compute $d_{k}$ for $2\leq k \leq s$, we note that:
\begin{equation}\label{eq:dk}
\begin{aligned}
d_k & = \sum_{3<i_2<\dots< i_k}\sum_{i_1=1}^3m_{\{i_1,\dots,i_k\}}+\sum_{3<i_3<\ldots<i_k}m_{\{1,2,i_3,\ldots,i_k\}}+\sum_{3<i_3<\dots<i_k}m_{\{1,3,i_3,\ldots,i_k\}} \\ & +\sum_{3<i_3<\dots < i_k}m_{\{2,3,i_3,\dots,i_k\}}+\sum_{3<i_4<\dots<i_k}m_{\{1,2,3,i_4,\dots,i_k\}},
\end{aligned}
\end{equation}
where the last summand is zero if $k=2$. 
We consider each term independently now.  We first consider the first summand coming from the terms $m_{\{i_1,\dots,i_{k}\}}$ with $i_1\in \{1,2,3\}$ and $i_j\notin \{1,2,3\}$ for $j={2,\ldots,k}$. By Proposition~\ref{Jacblock}, the matrices $M_{\{i_1,\dots,i_{k}\}}$  satisfy, for $j\geq 2$, that:
\begin{itemize}
\item The entry $(j,1)$ is zero if either $i_1=1$, or both $i_1=2$ and $\delta=0$. 
\item  The entry $(1,j)$ is zero if either $i_1=3$, or both $i_1=2$ and  $\delta=1$. 
\end{itemize}
Hence, 
 \[\sum_{3<i_2<\dots< i_k}\sum_{i_1=1}^3 m_{\{i_1,\ldots,i_k\}} =\left(\sum_{i_1=1}^3  m_{\{i_1\}}\right)\left( \sum_{i_2<\ldots<i_{k}}m_{\{i_2,\ldots,i_k\}} \right)= d_1 T_{k-1}.\]

To compute the other terms, we first notice that a direct computation shows that 
\[ m_{\{1,2\}}=0, \quad m_{\{1,3\}}=- (\l_m+L_{r-1})L_{r-2}h_1h_3,   \quad m_{\{2,3\}}=(1-\delta)(\l_m+L_{r-1})L_{r-2}h_2h_3,\]
which gives that 
	 \begin{align*}
	 d_2&=d_1T_1+m_{\{1,3\}}+m_{\{2,3\}}=d_1T_1 - (\l_m+L_{r-1})L_{r-2}h_1h_3+(1-\delta)(\l_m+L_{r-1})L_{r-2}h_2h_3.
	 \end{align*}

We focus now on $k\geq 3$. To compute $m_{\{1,3,i_3,\ldots,i_k\}}$, 
we add the first row of $M_{\{1,3,i_3,\ldots,i_k\}}$ to its second row (which is the  third row of $J$).  Then the first column of $M_{\{1,3,i_3,\ldots,i_k\}}$ is zero except for the first entry which is $-(\ell_m + L_{r-1})h_1$.  After this operation, the other entries of $M_{\{1,3,i_3,\ldots,i_k\}}$ where $\ell_m$ appears are on the first row, and hence computing the determinant of $M_{\{1,3,i_3,\ldots,i_k\}}$ via Laplace expansion along the first column gives that 
$m_{\{1,3,i_3,\ldots,i_k\}}$ is $(\l_m+L_{r-1})h_1h_3$ times a polynomial in $h_4,\dots,h_n,\l_1,\dots,\ell_{m-1}$. Therefore,
\[ \sum_{3<i_3<\ldots<i_k}m_{\{1,3,i_3,\ldots,i_k\}} = (\l_m+L_{r-1})h_1h_3   S_k^{(2)}, \]
for $S_k^{(2)}$ a polynomial in $h_4,\dots,h_n,\l_1,\dots,\ell_{m-1}$. 

To compute the remaining three summands of $d_k$ in \eqref{eq:dk}, we distinguish between $\delta=0$ and $\delta=1$. 
First, we let $\delta=1$. Then the second and third rows of $J(h,\ell)$ in Proposition~\ref{Jacblock} are linearly dependent and hence, 
\[ m_{\{2,3,i_3,\ldots,i_k\}}=0, \qquad m_{\{1,2,3,i_4,\ldots,i_k\}}=0. \] 
All we need is to compute $m_{\{1,2,i_3,\ldots,i_k\}}$. 
 If we subtract the first row of the submatrix $M_{\{1,2,i_3,\ldots,i_k\}}$ to the second, then  the first column  is zero except for the first entry which is $(\ell_m + L_{r-1})h_1$.  Arguing as above, this gives that 
\[ \sum_{3<i_3<\ldots<i_k}m_{\{1,2,i_3,\ldots,i_k\}} = (\l_m+L_{r-1})h_1h_2   S_k^{(1)}, \]
for $S_k^{(1)}$ a polynomial in $h_4,\dots,h_n,\l_1,\dots,\ell_{m-1}$. 

If $\delta=0$, the first two columns of $J(h,\ell)$ are identical and therefore,  for all $3<i_3<\dots<i_k$, 
\[  m_{\{1,2,i_3,\ldots,i_k\}}=m_{\{1,2,3,i_4,\ldots,i_k\}}=0.\]
To compute $m_{\{2,3,i_3,\ldots,i_k\}}$, adding the first row of  $M_{\{2,3,i_3,\ldots,i_k\}}$ to the second, and arguing as above, we obtain that 
\[ \sum_{3<i_3<\ldots<i_k}m_{\{2,3,i_3,\ldots,i_k\}} = (\l_m+L_{r-1})h_2h_3   S_k^{(3)}, \]
for $S_k^{(3)}$ a polynomial in $h_4,\dots,h_n,\l_1,\dots,\ell_{m-1}$. 

 Putting all this together, we have that for $k\geq 3$,
 \begin{align*}
d_k & =d_1 T_{k-1}  + \delta  (\l_m+L_{r-1})h_1h_2   S_k^{(1)} + (\l_m+L_{r-1})h_1h_3   S_k^{(2)}  + (1-\delta) (\l_m+L_{r-1})h_2h_3   S_k^{(3)} \\
&= d_1 T_{k-1}  +  (\l_m+L_{r-1}) \big(  \delta h_1h_2   S_k^{(1)} + h_1h_3   S_k^{(2)}  + (1-\delta) h_2h_3   S_k^{(3)}\big). 
\end{align*}
as stated. 
This concludes the proof.
\end{proof}

We now have all the ingredients to prove Theorem~\ref{thm:yes}. We prove a stronger version, where the inclusion of sets of reduced characteristic polynomials is made explicit. Then  Theorem~\ref{thm:yes} is a straightforward corollary.

\begin{theorem}\label{thm:charpol}
	Let $G$ be a network that satisfies assumptions (A1)-(A5) of Theorem~\ref{thm:yes}, and let $G'$ be the network obtained by removing the reverse reaction of the motif \eqref{motif}. Then,  there exists a map 
	\[\phi\colon \R_{> 0}^n\times \R_{> 0}^m\to\R_{> 0}^n\times \R_{> 0}^{m-1} 
	\]
	such that $p'_{\phi(h,\l)}=p_{(h,\l)}$ for all $(h,\l)\in \R_{> 0}^n\times \R_{> 0}^m$.
	In particular, 
	 \[\{p_{(h,\l)} :  (h,\l)\in \R^n_{>0}\times \R^m_{>0}\}\subseteq \{p'_{(h,\l')} :  (h,\l')\in \R^n_{>0}\times \R^{m-1}_{>0}\}. \]
\end{theorem}

\begin{proof}
We prove the statement by explicitly constructing the map  
\[ \phi(h,\ell)=(\phi_1(h,\l),\ldots,\phi_{n+m-1}(h,\l)). \]
To simplify the notation, throughout we denote
\[ h'= \phi(h,\ell)_{1:n}, \qquad \ell'=\phi(h,\ell)_{n+1:n+m-1}.  \]
As $p'_{\phi(h,\l)}=  p'_{(h',\l')}= p_{(h',\iota(\ell'))}$ by Proposition~\ref{prop:jac}, all we need is to show that 
\begin{equation*}
p_{(h,\ell)} =  p_{(h', \iota(\ell'))}. 
\end{equation*}
We define the last $n+m-4$ coordinates of $\phi$ to be the identity in the sense that 
\begin{equation}\label{eq:id}
\begin{aligned}
h_i'= \phi_i(h,\l)&:= h_i  &  \text{ for }&  i=4,\ldots n,\\
\ell_i' = \phi_{i+n}(h,\l)&:= \l_{i} &  \text{ for } & i=1,\ldots m-1.
\end{aligned}
\end{equation}	
As $\beta_{(h_{4:n},\l)}$ in \eqref{eq:psplitting} does not depend on $h_1,h_2,h_3$ nor on $\ell_m$ by Proposition~\ref{prop:charpolycoeffG}, \eqref{eq:id} readily gives that
\[\beta_{(h,\ell)} =  \beta_{(h', \iota(\ell'))}. \]
Therefore, we need to define $h_1'=\phi_1,h_2'=\phi_2,h_3'=\phi_3$ such that 
\begin{equation}\label{eq:alpha2}
\alpha_{(h,\ell)} =  \alpha_{(h', \iota(\ell'))}. 
\end{equation}
By \eqref{eq:id} again, the minors $T_i$ and the polynomials $S_k^{(1)}$, $S_k^{(2)}$, $S_k^{(3)}$  from Proposition~\ref{prop:charpolycoeffG} are mapped to themselves by $\phi$, as they do not depend on $h_1,h_2,h_3,\ell_m$. 
Similarly, recalling that in the description of $J(h,\ell)$ from Proposition~\ref{Jacblock} we consider $L:=E'\pi(\ell)=E'\ell'$, $L$ is mapped to itself by $\phi$. 
This holds independently of how we define $h_1',h_2',h_3'$. 

As the last entry of $\iota(\ell')$ is zero, Proposition~\ref{prop:charpolycoeffG} now gives that  \eqref{eq:alpha2} holds if for
\[ \Lambda :=  (\l_m+L_{r-1}) h_1+ (\l_m+L_{r-1})h_2 + (\l_m+L_{r-2})h_3,\] 
$h_1',h_2',h_3'$ satisfy    that
\begin{equation}\label{eq:tosee}
\begin{aligned}
\Lambda &=  L_{r-1} h'_1 +L_{r-1}h'_2 +L_{r-2}h'_3 \\
(\ell_m+L_{r-1})h_1h_3 &= L_{r-1}h'_1h'_3  \\ 
(\ell_m+ L_{r-1})h_2h_3 &= L_{r-1}h'_2h'_3 \quad (\text{if }\delta=0)\\
 (\l_m+L_{r-1})  h_1h_2  &=  L_{r-1}  h'_1h'_2   \quad (\text{if } \delta=1). 
\end{aligned}
\end{equation}

In order to define $h_1',h_2',h_3'$,  we separate the cases $\delta=1$ and $\delta=0$. 
If $\delta=1$, we define 
	 {\small	\begin{align*}
		h_1'=\phi_1(h,\l)&:=\frac{\Lambda-\sqrt{\gamma_1}}{2L_{r-1}},\qquad
		h_2'=\phi_2(h,\l):=\frac{(\Lambda+\sqrt{\gamma_1})h_2}{2\gamma_2},\qquad
		h_3'=\phi_3(h,\l):=\frac{(\Lambda +\sqrt{\gamma_1})h_3}{2\gamma_2},
		\end{align*}}%
where
\begin{align*}
		\gamma_1&:=\Lambda^2-4h_1 (\l_m + L_{r-1}) \gamma_2,  &
		\gamma_2&:= h_2L_{r-1}+h_3L_{r-2}. 
	\end{align*}
With this, we have 
 {\small\begin{align*}
L_{r-1}  h'_1h'_j &= \frac{L_{r-1}(\Lambda-\sqrt{\gamma_1})(\Lambda+\sqrt{\gamma_1})h_j}{4 L_{r-1} \gamma_2}
= \frac{(\Lambda^2-\gamma_1)h_j}{4 \gamma_2}  = h_1 h_j(\l_m + L_{r-1}), \quad j=2,3, 
\end{align*}}%
and 
 {\small\begin{align*}
L_{r-1} (h'_1 +  h'_2) +L_{r-2}h'_3 &= 
\frac{\Lambda-\sqrt{\gamma_1}}{2} + \frac{(\Lambda+\sqrt{\gamma_1})(h_2 L_{r-1} + h_3 L_{r-2})}{2\gamma_2} = \frac{(\Lambda-\sqrt{\gamma_1}) + (\Lambda+\sqrt{\gamma_1}) }{2 }  = \Lambda,
\end{align*}}%
as desired. 	So \eqref{eq:tosee} holds for $\delta=1$ with this definition of $\phi$.
 
\medskip
We now consider $\delta=0$, and define similarly
 {\small
 	\begin{align}\label{eq:phidelta0}
 h_1'=	\phi_1(h,\l)&:=\frac{(\Lambda+\sqrt{\gamma_1})h_1}{2\gamma_2},\qquad
 h_2'= 	\phi_2(h,\l):=\frac{(\Lambda+\sqrt{\gamma_1})h_2}{2\gamma_2},\qquad
  h_3'=	\phi_3(h,\l):=\frac{\Lambda-\sqrt{\gamma_1}}{2L_{r-2}},
 	\end{align}
 }%
 where 
 {\small
 	\begin{align*}
 	\gamma_1&:=  \Lambda^2-4 \frac{L_{r-2}}{L_{r-1} }\gamma_2 h_3 (\l_m+L_{r-1}) & 
 	\gamma_2&:=L_{r-1} (h_1+h_2).
 	\end{align*}
 }
With this, we have  for  $j=1,2$:
 {\small\begin{align*}
L_{r-1}  h'_jh'_3 &= \frac{L_{r-1}(\Lambda^2-\gamma_1)h_j}{4L_{r-2} \gamma_2}
= h_j h_3(\l_m + L_{r-1}).
\end{align*}}%
Finally, 
 {\small\begin{align*}
L_{r-1} (h'_1 +  h'_2) +L_{r-2}h'_3 &= 
 \frac{(\Lambda +\sqrt{\gamma_1})(h_1+h_2)L_{r-1}}{2\gamma_2}  + \frac{(\Lambda-\sqrt{\gamma_1})L_{r-2} }{2 L_{r-2}} 
 =  \frac{\Lambda-\sqrt{\gamma_1} + \Lambda+\sqrt{\gamma_1} }{2 }  = \Lambda.
 \end{align*}}%
 So \eqref{eq:tosee} holds for $\delta=0$ with this definition of $\phi$ and this concludes the proof.
\end{proof}

\begin{example}\label{ex:constructive}
We revisit Example~\ref{network:processivedistributive}. Based on parameter values reported in \cite{CMS}, the network 
in  \eqref{eq:network5} displays a pair of purely imaginary eigenvalues for
{\small \begin{align*}
	(\k_1,\ldots,\k_{10})& =(1,1,1,3,1,100,1,100,1,0.9)\\
	(x_1,\ldots,x_9)& \approx (4.22647, 0.065134, 2.08295, 22.752, 1.42669, 4.40176, 0.0440176,
	4.40176, 4.89085),
	\end{align*}}%
 if species are ordered as $E,F,S_0,S_1,S_2,ES_0,FS_1,ES_1,FS_1$. In terms of the convex parameters with extreme matrix 
{\small	\[E= \begin{bmatrix}
        1 & 0 & 1 & 1 & 0 & 1 & 1 & 1 & 0& 1\\
        0 & 0 & 0 & 0 & 0 & 0 & 0 & 1 & 1& 0\\
	0 & 0 & 0 & 1 & 1 & 0 & 0 & 0 & 0& 0 \\
	1 & 1 & 0 & 0 & 0 & 0 & 0 & 0 & 0& 0
	\end{bmatrix}^\top,\]}%
	 these values correspond to
{\small \begin{align*}
(h_1,\ldots,h_{9}) & \approx (0.236604, 15.353, 0.480089, 0.0439522, 0.700923, 0.227182, 22.7182,
0.227182, 0.204464), \\
(\l_1,\ldots,\l_4) & \approx (4.40176, 4.89085, 0.0440176,4.40176).
\end{align*}}%
Following the construction of $\phi$ in the proof of Theorem~\ref{thm:charpol} in \eqref{eq:id} and \eqref{eq:phidelta0}, for the reduced network~\eqref{eq:network6}, convex parameters $\phi(h,\ell)=(h',\ell')\in \R^9_{>0}\times \R^3_{>0}$ for which $J'(h',\ell')$ has a pair of purely imaginary eigenvalues  are
{\small \begin{align*}
(h'_1,\ldots,h'_{9}) & \approx (0.559811, 15.353, 1.1359, 0.0439522, 0.700923, 0.192037, 22.7182, \
0.227182, 0.204464)\\
(\l_1',\l_2',\l'_3) & \approx (4.40176, 4.89085, 0.0440176),
\end{align*}}%
which correspond to the following reaction rate constants and steady state:
{\small\begin{align*}
(\k_1,\k_3,\ldots,\k_{10}) & \approx (2.79903,0.8453,3,1,100,1, 100, 1, 0.9)\\
(x_1,\ldots,x_9) & \approx (1.78632, 0.065134, 0.880358, 22.752, 1.42669, 5.20734, 0.0440176, \
4.40176, 4.89085).
\end{align*}}%
By construction, the choices of reaction rate constants and steady states for \eqref{eq:network5} and \eqref{eq:network6} give rise to the same characteristic polynomial, and hence have the same roots. 
\end{example}

{\small 

}

\end{document}